\RequirePackage{fix-cm}
\RequirePackage{amsmath}
\documentclass{amsart}
\usepackage{graphicx}
\usepackage{anyfontsize}
\usepackage{babel}
\usepackage{amssymb,enumerate,ifthen}
\usepackage[mathscr]{eucal}

\usepackage{xcolor}
\usepackage[T1]{fontenc}
\usepackage[utf8]{inputenc}
\usepackage[hmargin={22mm,22mm},vmargin={25mm,25mm}]{geometry}

\newcommand{\N}{{\mathbb N}}
\newcommand{\R}{{\mathbb R}}

\newcommand{\E}{\mathbb{E}\,}

\DeclareMathAlphabet\mathbfcal{OMS}{cmsy}{b}{n}
\mathchardef\mhyphen="2D 
\newtheorem{Thm}{Theorem}
\newtheorem{Lem}[Thm]{Lemma}
\newtheorem{Cor}[Thm]{Corollary}
\newtheorem{Rem}[Thm]{Remark}
\newtheorem{Exa}{Example}
\newtheorem{defin}{Definition}
\newtheorem{prop}{Proposition}

\begin{document}

\title{On Popoviciu's concept of convexity for functions of $d$ variables}

\author{Andrzej Komisarski}
\address[Andrzej Komisarski]{Faculty of Mathematics and Computer Science, University of Lodz, Banacha 22, 90--238 Łódź, Poland}
\email{\color{blue}andkom@math.uni.lodz.pl}

\author{Teresa Rajba}
\address[Teresa Rajba]{Department of Mathematics, University of Bielsko-Biala, ul.\ Wil\-lo\-wa~2, 43-309 Bielsko-Bia{\l}a, Poland}
\email{\color{blue}trajba@ubb.edu.pl}

\begin{abstract}
{We establish an integral representation for Popoviciu’s convex functions of $d$ variables. This representation serves as a~foundation for deriving several functional inequalities, analogous to those well-known for usual convex functions. Our results generalize and extend the results obtained by S.~Gal, C.~Niculescu, B.~Gavrea, T.~Popoviciu, and others, who considered only differentiable functions of two variables. In contrast to other authors, we do not impose any additional regularity assumptions on the studied functions.}

\keywords{divided differences \and inequalities related to stochastic convex orderings \and box-convex functions \and Hermite-Hadamard and Jensen inequalities \and functional inequalities related to convexity}
\subjclass{26D05 \and 39B62 \and 60E15}
\end{abstract}

\maketitle
\section{Introduction}\label{intro}
H. Hopf \cite{Hopf1926} and T. Popoviciu \cite{Popoviciu1934,Popoviciu1944} introduced the notion of higher order convex functions based on the so called divided differences. Let $n\geq1 $. 
Let  $I$ be a subinterval of $\R$. Given a function $f\colon I\to\mathbb R$ of one real variable and a system $x_0, x_1, \ldots, x_{n+1}$ of pairwise distinct points of $I$ the \emph{divided differences} of order $0,1, \ldots, n+1$ are respectively defined by the formulas 
\begin{align*}
\left[ x_0;f \right]&=f(x_0),\\
\left[ x_0, x_1;f \right]&=\frac{f(x_1)-f(x_0)}{x_1-x_0},\quad\ldots,\\
\left[ x_0, x_1,\ldots,x_{n+1};f \right]&=\frac{\left[ x_1, x_2,\ldots, x_{n+1};f \right]-\left[ x_0, x_1,\ldots, x_{n};f \right]}{x_{n+1}-x_0}.
\end{align*}
The function $f(x)$ is called \emph{$n$-convex (resp. $n$-concave, $n$-affine)}, if $[ x_0, x_1,\ldots,$ $x_{n+1};f]\geq 0$ (resp. $\leq0$, $=0$) for any pairwise distinct points $x_0, x_1,\ldots,x_{n+1}$. The function $f$ is convex when all divided differences of order two are nonnegative for all systems of~pairwise distinct points. If $f^{(n+1)} $ exists, then $f$ is \emph{$n$-convex ($n$-concave)} if, and only if, $f^{(n+1)} \geq 0$ ($f^{(n+1)} \leq 0$).
\begin{prop}[\cite{Popoviciu1934}]\label{prop:1.1}
A function $f(x)$ is $n$-convex if and only if its derivative $f^{(n-1)}(x)$  exists and is convex (with the convention $f^{(0)}(x)=f(x)$).
\end{prop}

If $f=f(x,y)$ is a function defined on a rectangle $I \times J$ and $x_0, x_1,\ldots,x_{m}$ are pairwise distinct points of $I$ and $y_0,y_1,\ldots,y_n$ are pairwise distinct points of $J$, one defines the \emph{divided double difference} of order $(m,n)$ by the formula 
\begin{align}
\left[\LARGE{\substack{{x_0,\quad x_1, \quad \ldots \quad x_m}\\\\{y_0, \quad y_1, \quad \ldots \quad y_n}}};f \right]&= \left[ x_0, x_1,\ldots, x_{m};\left[ y_0, y_1,\ldots,y_{n};f(x, \cdot) \right]\right]\\
&= \left[  y_0, y_1,\ldots,y_{n};\left[ x_0, x_1,\ldots, x_{m};f(\cdot,y) \right]\right]
\end{align}

Drawing a parallel to the one dimensional case, T. Popoviciu \cite{Popoviciu1934}, p.78, has called a function $f\colon I \times J  \to\mathbb R$ \emph{convex of order $(m,n)$} (\emph{box-$(m,n)$-convex} in our terminology) if all divided differences 
\begin{equation*}
\left[\LARGE{\substack{{x_0,\quad x_1, \quad \ldots \quad x_m}\\\\{y_0, \quad y_1, \quad \ldots \quad y_n}}};f \right]
\end{equation*}
are nonnegative for all pairwise distinct points $x_0, x_1,\ldots, x_{m}$  and $y_0, y_1,\ldots, y_n$. 
The related notions of \emph{box-$(m,n)$-concave} function and \emph{box-$(m,n)$-affine} fun\-cti\-on can be introduced in the standard way.

In \cite{KomRaj2023}, we obtained the integral representation, the Ra\c{s}a, Jensen and Hermite-Hadamard inequalities for box-$(m,n) $-convex functions.

Let  $\textbf{n}=(n_1,\ldots,n_d)$, $d\geq 1$, $n_1,\ldots,n_d\in\mathbb N$. In this paper, we define and study  box-$\textbf{n}$-convex functions of $d$ real variables, such that in the case $d=2$, $n_1=m $, $n_2=n $, the class of box-$\textbf{n}$-convex functions coincides with the class of  box-$(m,n)$-convex functions. We give the integral representation and several characterizations of box-$\textbf{n}$-convex fun\-cti\-ons. Using the integral representation, we obtain the Ra\c{s}a, Jensen and Hermite-Hadamard inequalities for box-$\textbf{n} $-convex functions.

The results known in the literature regarding box-convex functions used differential methods to study these functions. They required assumptions about high regularity (the existence of derivatives of high orders).
In this article, we do not make any assumptions about regularity of the studied functions (they may even be non-measurable). However, we still use differential methods. This is possible thanks to the notion of $\mathbf n$-regularity introduced in Section 5 and the regularization of box-$\mathbf n$-convex functions, based on Proposition~\ref{prop:regularize}.
We first used this idea in \cite{KomRaj2023} for functions of two variables. However, the general case (with $d>2$) is much more complicated and introduces some phenomena that do not occur for $d=2$.

In Section 2, we recall the selected properties of divided differences of order $n$ given in \cite{KomRaj2023}.

In Section 3, we give the definitions and selected properties of multiple divided differences of order $n$.

In Section 4, we prove that the set of box-$\textbf{n}$-affine  functions coincides with the set of pseudopolynomials of degree  $(n_1-1, \ldots, n_d-1)$.

In Section 5, we introduce the notion of $\textbf{n}$-regularity of functions and we give properties of $\textbf{n}$-regular functions.

In Section 6, we study properties of integration and differentation  of $\textbf{n}$-regular functions.

In Section 7, we obtain the integral representation of box-monotone function, i.e. box-$(1,\ldots, 1)$-convex functions. 

In Section 8, we obtain the integral representation of box-$\textbf{n}$-convex function.

In Section 9, we give properties of box-$\textbf{n}$-convex orders.

In Section 10, we obtain the Hermite-Hadamard, Jensen and Ra\c{s}a inequali\-ties for box-$\textbf{n}$-convex functions.

Throughout this article, we assume that $\mathbb N=\{0,1,2,\dots\}$ unless otherwise stated.

\section{Selected properties of divided differences of order $n$}

The following expanded form of divided difference is well known.
\begin{prop}\label{prop:1.2}
\begin{equation}\label{eq:expanded}
\left[x_0,x_1,\ldots,x_n;f \right]
=\sum_{j=0}^n\frac{f(x_j)}{\prod\limits_{\substack{l=0\\l\neq j}}^n(x_j-x_l)}.
\end{equation}
\end{prop}

\begin{Lem}
\label{lem:il.podz.}
Let $n\geq1$, let $I\subset \mathbb R$ be an~interval, $\alpha\in I$ and let $f\colon I\to\mathbb R$ be an~integrable function.
For pairwise distinct points $x_0,x_1,\dots,x_n\in I$, we have
\begin{equation*}
\left[x_0,x_1,\dots,x_n;\int_\alpha^\cdot f(t)dt\right]=\int_0^1t^{n-1}[x_{1,t},\dots,x_{n,t};f]dt,
\end{equation*}
where $x_{i,t}=tx_i+(1-t)x_0$ for $t\in[0,1]$ and $i=1,2,\dots,n$.
\end{Lem}
\begin{proof}
The proof of the lemma is by induction on $n$. For $n=1$ we have
$$\left[x_0,x_1;\int_\alpha^\cdot f(t)dt\right]=\frac{\int_{x_0}^{x_1} f(t)dt}{x_1-x_0}=\int_0^1f(x_{1,t})dt=\int_0^1t^{1-1}[x_{1,t};f]dt.$$
In the induction step, we use the definition of the divided difference and the fact that it does not depend on the permutation of the points $x_0,\dots,x_{n+1}$.
\begin{multline*}
\left[x_0,x_1,\dots,x_{n+1};\int_\alpha^\cdot f(t)dt\right]
=\left[x_1,x_0,x_2,\dots,x_{n+1};\int_\alpha^\cdot f(t)dt\right]\\
=\frac{\left[x_0,x_2,\dots,x_{n+1};\int_\alpha^\cdot f(t)dt\right]-\left[x_0,x_1,\dots,x_n;\int_\alpha^\cdot f(t)dt\right]}{x_{n+1}-x_1}\\
=\int_0^1t^{n-1}\frac{[x_{2,t},\dots,x_{n+1,t};f]-[x_{1,t},\dots,x_{n,t};f]}{x_{n+1}-x_1}dt\\
=\int_0^1t^n\frac{[x_{2,t},\dots,x_{n+1,t};f]-[x_{1,t},\dots,x_{n,t};f]}{x_{n+1,t}-x_{1,t}}dt\\
=\int_0^1t^{(n+1)-1}[x_{1,t},\dots,x_{n+1,t};f]dt.
\end{multline*}
\end{proof}

\begin{Lem}
\label{lem:pochodna}
Let $n\geq2$, let $I\subset\mathbb R$ be an interval, and let $f\colon I\to\mathbb R$ be a right-differentiable function.
Let $f'_R\colon I\setminus\{\sup I\}\to\mathbb R$ be a~right-derivative of $f$.
For pairwise distinct points $x_1,\dots,x_n\in I\setminus\{\sup I\}$ and $k=1,\dots,n$, the limit $\lim_{x_0\downarrow x_k}[x_0,x_1,\dots,x_n;f]$ exists and
\begin{equation}\label{eq:poch}
\lim_{x_0\downarrow x_k}[x_0,x_1,\dots,x_n;f]
=\frac{f'_R(x_k)}{\prod\limits_{\substack{i=1\\i\neq k}}^n(x_k-x_i)}+\sum_{\substack{j=1\\j\neq k}}^n\frac1{x_j-x_k}\left(\frac{f(x_j)}{\prod\limits_{\substack{i=1\\i\neq j}}^n(x_j-x_i)}+\frac{f(x_k)}{\prod\limits_{\substack{i=1\\i\neq k}}^n(x_k-x_i)}\right).
\end{equation}
Moreover,
$$[x_1,\dots,x_n;f'_R]=
\sum_{k=1}^n\lim_{x_0\downarrow x_k}[x_0,x_1,\dots,x_n;f].$$
\end{Lem}
\begin{proof}
Due to symmetry, it is enough to show that \eqref{eq:poch} holds for $k=1$. We skip an easy proof by induction on $n$.
Using \eqref{eq:poch}, we obtain $$\sum_{k=1}^n\lim_{x_0\downarrow x_k}[x_0,x_1,\dots,x_n;f]=\sum_{k=1}^n\frac{f'_R(x_k)}{\prod\limits_{\substack{i=1\\i\neq k}}^n(x_k-x_i)}=[x_1,\dots,x_n;f'_R].$$
\end{proof}

\section{Divided differences of order $(n_1,\ldots, n_d)$ and box-$(n_1,\ldots, n_d)$-convex functions}

Let $d\in\mathbb N$. For every $A\subset\{1,2,\dots,d\}$ let $A'=\{1,2,\dots,d\}\setminus A$ be the complement of the set $A$.

For $\mathbf x\in\mathbb R^d$, $\mathbf n\in\mathbb N^d$ and $A\subset\{1,2,\dots,d\}$ let $\mathbf x_A=(x_{a_1},x_{a_2},\dots,x_{a_{|A|}})$, and $\mathbf n_A=(n_{a_1},n_{a_2},\dots,n_{a_{|A|}})$, where $(a_1,a_2,\dots,a_{|A|})$ is the ordered sequence of the elements of $A$ (i.e., $A=\{a_1,a_2,\dots,a_{|A|}\}$ and $a_1<a_2<\dots<a_{|A|}$). In particular $\mathbf x_{\{i\}}=(x_i)$ and $\mathbf x_{\{i\}'}=(x_1,\dots,x_{i-1},$ $x_{i+1},\dots,x_d)$.

Let $I_1,I_2,\dots,I_d\subset\mathbb R$ be intervals (bounded or unbounded). We denote $\mathbf I=\prod_{i=1}^dI_i\subset\mathbb R^d$. If $A\subset\{1,2,\dots,d\}$, then we put $\mathbf I_A=\prod_{i\in A}I_i\subset\mathbb R^{|A|}$.
Clearly, if $\mathbf x\in\mathbf I$, then $\mathbf x_A\in\mathbf I_A$.

For $f\colon\mathbf I\to\mathbb R$, $A\subset\{1,2,\dots,d\}$ and $\mathbf z\in\mathbf I_{A'}$ we denote by $f_A^{\mathbf z}=f_A\colon\mathbf I_A\to\mathbb R$ the function given by
$$f_A^{\mathbf z}(\mathbf y)=f(\mathbf x),$$
where
$$x_i=
\begin{cases}
y_j&\text{if $i\in A$ and } a_j=i,\\
z_j&\text{if $i\in A'$ and }a'_j=i,
\end{cases}$$
$(a_1,a_2,\dots,a_{|A|})$ is the ordered sequence of the elements of $A$, and $(a'_1,a'_2,$ $\dots,a'_{|A'|})$ is the ordered sequence of the elements of $A'$
(i.e., $\mathbf x$ is the unique element of $\mathbf I$ satisfying $\mathbf x_A=\mathbf y$ and $\mathbf x_{A'}=\mathbf z$).
In particular, $f_A^{\mathbf x_{A'}}(\mathbf x_A)=f(\mathbf x)$.
The idea behind the notation $f_A^\cdot(\cdot)$ is to reduce the number of variables of $f$ by treating some of them as parameters (those that are standing at the positions described by the set $A'$).

It is convenient to consider the special case $d=0$. Then $f\colon\mathbf I\to\mathbb R$ is a~constant (a function of $0$ variables).


\begin{defin}\label{def:multdif}
Let $f\colon\mathbf I\to\mathbb R$ and $\mathbf n=(n_1,n_2,\dots,n_d)\in\mathbb N^d$. For $i=1,2,\dots,d$ let $\mathbf x_i=(x_{i0},x_{i1},\dots,x_{in_i})\in I_i^{n_i+1}$ be a vector with pairwise distinct coordina\-tes. We inductively define the multiple divided difference of order $\mathbf n$ as follows.

If $d=0$ (i.e. $f$ is a constant), then $[;f]=f$.

If $d=1$, then $[\mathbf x_1;f]=[x_{10},x_{11},\dots,x_{1n_1};f]$ is the divided difference defined in the Introduction (Section~\ref{intro}).

If $d>1$, then
$$\left[\ \LARGE{\substack{\mathbf x_1\\\\\mathbf x_2\\\\\dots\\\\\mathbf x_d}};f\right]=\left[\LARGE{\substack{x_{10},\ x_{11},\ldots,\ x_{1n_1}\\\\x_{20},\ x_{21},\ldots,\ x_{2n_2}\\\\\dots\\\\x_{d0},\ x_{d1},\ldots,\ x_{dn_d}}};f\right]=
[x_{d0},x_{d1},\ldots,\ x_{dn_d};g],$$
where $g\colon I_d\to\mathbb R$ is given by
$$g(x)=\left[\LARGE{\substack{x_{10},\ x_{11},\ldots,\ x_{1n_1}\\\\x_{20},\ x_{21},\ldots,\ x_{2n_2}\\\\\dots\\\\x_{d-1,0},\ x_{d-1,1},\ldots,\ x_{d-1,n_{d-1}}}};f_{\{1,2,\dots,d-1\}}^x\right].$$
\end{defin}

In other words, we apply the divided differences to successive arguments of $f$.

In the following proposition, we give the expanded form of divided difference of order $\textbf{n}$, which is a generalization of \eqref{eq:expanded}.
\begin{prop}\label{prop:prop5}
If $d>0$, then
\begin{equation}\label{eq:prop5}
\left[\LARGE{\substack{x_{10},\ x_{11},\ldots,\ x_{1n_1}\\\\x_{20},\ x_{21},\ldots,\ x_{2n_2}\\\\\dots\\\\x_{d0},\ x_{d1},\ldots,\ x_{dn_d}}};f\right]=\sum_{j_1=0}^{n_1}\sum_{j_2=0}^{n_2}\dots\sum_{j_d=0}^{n_d}\frac{f(x_{1j_1},x_{2j_2},\dots,x_{dj_d})}{\prod\limits_{i=1}^d\prod\limits_{\substack{l_i=0\\l_i\neq j_i}}^{n_i}(x_{ij_i}-x_{il_i})}.
\end{equation}
\end{prop}
\begin{proof}[Sketch of the proof]
By induction on $m=1,2,\dots,d$ and taking into account  Proposition~\ref{prop:1.2}, we obtain
$$
\left[\LARGE{\substack{x_{10},\ x_{11},\ldots,\ x_{1n_1}\\\\x_{20},\ x_{21},\ldots,\ x_{2n_2}\\\\\dots\\\\x_{m0},\ x_{m1},\ldots,\ x_{mn_m}}};g\right]=\sum_{j_1=0}^{n_1}\sum_{j_2=0}^{n_2}\dots\sum_{j_m=0}^{n_m}\frac{g(x_{1j_1},x_{2j_2},\dots,x_{mj_m})}{\prod\limits_{i=1}^m\prod\limits_{\substack{l_i=0\\l_i\neq j_i}}^{n_i}(x_{ij_i}-x_{il_i})},
$$
for every function $g\colon\mathbf I_{\{1,2,\dots,m\}}\to\mathbb R$ such that
$g=f_{\{1,2,\dots,m\}}^{\mathbf y}$ for some $\mathbf y\in I_{\{m+1,\dots,d\}}$. For $m=d$ we obtain \eqref{eq:prop5}.
\end{proof}

By symmetry of formula \eqref{eq:prop5}, we see that the order in which the divided differences are applied in Definition~\ref{def:multdif} is not important
(similarly, as in the two-dimensional case, cf. \cite{GalNic2019,KomRaj2023,Popoviciu1934}). In particular, the following observation holds.

\begin{Rem}\label{rem:rem1}
If $A\subset\{1,2,\dots,d\}$, $A=\{a_1,a_2,\dots,a_{|A|}\}$,  $a_1<a_2<\dots<a_{|A|}$, and $A'=\{a'_1,a'_2,\dots,a'_{|A'|}\}$,  $a'_1<a'_2<\dots<a'_{|A'|}$, then
$$\left[\ \LARGE{\substack{\mathbf x_1\\\\\mathbf x_2\\\\\dots\\\\\mathbf x_d}};f\right]=
\left[\ \LARGE{\substack{\mathbf x_{a'_1}\\\\\mathbf x_{a'_2}\\\\\dots\\\\\mathbf x_{a'_{|A'|}}}};g\right],$$
where $g\colon\mathbf I_{A'}\to\mathbb R$ is given by
$$g(\mathbf x)=
\left[\LARGE{\substack{\mathbf x_{a_1}\\\\\mathbf x_{a_2}\\\\\dots\\\\\mathbf x_{a_{|A|}}}};f_A^{\mathbf x}\right].$$
\end{Rem}
An immediate consequence of Remark \ref{rem:rem1} is the following observation.

\begin{Rem}\label{rem:rem2}
Let $A\subset\{1,2,\dots,d\}$, $A=\{a_1,a_2,\dots,a_{|A|}\}$,  $a_1<a_2<\dots<a_{|A|}$, and $A'=\{a'_1,a'_2,\dots,a'_{|A'|}\}$,  $a'_1<a'_2<\dots<a'_{|A'|}$.
Let $f\colon\mathbf I\to\mathbb R$ be the function given by the formula $f(\mathbf x)=g(\mathbf x_A)\cdot h(\mathbf x_{A'})$, where $g\colon\mathbf I_A\to\mathbb R$ and $h\colon\mathbf I_{A'}\to\mathbb R$. Then
$$\left[\ \LARGE{\substack{\mathbf x_1\\\\\mathbf x_2\\\\\dots\\\\\mathbf x_d}};f\right]=
\left[\ \LARGE{\substack{\mathbf x_{a_1}\\\\\mathbf x_{a_2}\\\\\dots\\\\\mathbf x_{a_{|A|}}}};g\right]\cdot\left[\ \LARGE{\substack{\mathbf x_{a'_1}\\\\\mathbf x_{a'_2}\\\\\dots\\\\\mathbf x_{a'_{|A'|}}}};h\right].$$
\end{Rem}

\begin{defin}
Let $\mathbf n\in\mathbb N^d$. We say that the function $f\colon\mathbf I\to\mathbb R$ is \emph{box-$\mathbf n$-convex} if
\begin{equation}\label{eq:boxdef}
\left[\ \LARGE{\substack{\mathbf x_1\\\\\mathbf x_2\\\\\dots\\\\\mathbf x_d}};f\right]\geq0
\end{equation}
for every $\mathbf x_1,\mathbf x_2,\dots,\mathbf x_d$, where $\mathbf x_i=(x_{i0},x_{i1},\dots,x_{in_i})\in I_i^{n_i+1}$ is a vector with pairwise distinct coordinates ($i=1,2,\dots,d$).

We say that $f$ is \emph{box-$\mathbf n$-concave} if it satisfies \eqref{eq:boxdef} with $\geq$ replaced by $\leq$.

We say that $f$ is \emph{box-$\mathbf n$-affine} if \eqref{eq:boxdef} becomes the equality.
\end{defin}

\begin{Rem}
Observe that for every $\mathbf n\in\mathbb N^d$ the sets of all box-$\mathbf n$-convex functions and all box-$\mathbf n$-concave functions are convex cones in the space of all real functions on $\mathbf I$.
The set of all box-$\mathbf n$-affine functions is a~linear subspace of that space.
\end{Rem}

\begin{Rem}\label{rem:onedim}
The case $d=2$ was investigated in \cite{KomRaj2023}.

If $f\colon I\to\mathbb R$ is a function of one variable ($d=1$), then $f$ is box-$(n)$-convex if and only if it is $(n-1)$-convex in the classical sense.
In particular, box-$(0)$-convexity of a function $f$ means that $f$ is non-negative.
Box-$(1)$-convexity of $f$ means that $f$ is non-decreasing.
Box-$(2)$-convexity of a function is equivalent to its convexity.
If $f\colon I\to\mathbb R$ is box-$(2)$-convex and $I\subset\mathbb R$ is open, then $f$ is continuous and right-differentiable. For $n>2$ and open $I$ every box-$(n)$-convex function has a derivative of order $n-2$, which is a convex function.
\end{Rem}

By Remark \ref{rem:rem2} we obtain the following remark.
\begin{Rem}\label{rem:rem4}
Let $\mathbf n\in\mathbb N$ and $A\subset\{1,2,\dots,d\}$.
Let $f\colon\mathbf I\to\mathbb R$ be the function given by the formula $f(\mathbf x)=g(\mathbf x_A)\cdot h(\mathbf x_{A'})$, where $g\colon\mathbf I_A\to\mathbb R$  and $h\colon\mathbf I_{A'}\to\mathbb R$. Then
\begin{itemize}
\item[(a)]
if $g$ is box-$\mathbf n_A$-convex and $h$ is box-$\mathbf n_{A'}$-convex, then $f$ is box-$\mathbf n$-convex.
\item[(b)]
if $g$ is a linear combination of box-$\mathbf n_A$-convex functions and $h$ is a linear combination of box-$\mathbf n_{A'}$-convex functions, then $f$ is a linear combination of box-$\mathbf n$-convex functions.
\item[(c)]
if $g$ is box-$\mathbf n_A$-affine or $h$ is box-$\mathbf n_{A'}$-affine, then $f$ is box-$\mathbf n$-affine.
\end{itemize}
\end{Rem}

\section{Pseudo-polynomials and box-$\mathbf n$-affine functions}

\begin{defin}
Let $\mathbf n=(n_1,n_2,\dots,n_d)\in\{-1,0,\dots\}^d$. We say that the function $W\colon\mathbf I\to\mathbb R$ is a \emph{pseudo-polynomial of degree} $\mathbf n$ if it is of the form
$$W(x_1,x_2,\dots,x_d)=\sum_{i=1}^d\sum_{k=0}^{n_i}A_{ik}(\mathbf x_{\{i\}'})x_i^k,$$
where $\mathbf x=(x_1,x_2,\dots,x_d)$, and $A_{ik}\colon\mathbf I_{\{i\}'}\to\mathbb R$ ($i=1,2,\dots,d$ and $k=0,1,\dots,n_i$) are arbitrary functions (as usually, $\sum_{k=0}^{-1}A_{ik}(\mathbf x_{\{i\}'})x_i^k=0$ if $n_i=-1$).
\end{defin}

\begin{Exa}
Let $d=3 $ and $\mathbf n=(n_1,n_2,n_3)\in\mathbb N^3$. Then, the function $W\colon I_1 \times I_2 \times I_3\to\mathbb R$ is a \emph{pseudo-polynomial} of degree $\mathbf n$ if it is of the form
$$W(x_1,x_2,x_3)=\sum_{k=0}^{n_1}A_{1k}(x_2,x_3)x_1^k
+\sum_{k=0}^{n_2}A_{2k}(x_1,x_3)x_2^k
+\sum_{k=0}^{n_3}A_{3k}(x_1,x_2)x_3^k,$$
where $A_{1k}\colon I_2 \times I_3\to\mathbb R$  
 $(k=0,1,\dots,n_1)$, $A_{2k}\colon I_1 \times I_3\to\mathbb R$ $(k=0,1,\dots,n_2)$ and $A_{3k}\colon I_1 \times I_2\to\mathbb R$ 
 $(k=0,1,\dots,n_3)$ are arbitrary functions.
\end{Exa}

\begin{Rem}
Note that, in general, the degree of a pseudo-polynomial is not uniquely determined. For example $W(x,y)=e^xy+e^yx+x^2y^2$ is a pseudo-polynomial of both degrees $(1,2)$ and $(2,1)$, but it is not a pseudo-polynomial of degree $(1,1)$.
\end{Rem}

\begin{Lem}\label{lem:pseudo1}
Let $f\colon\mathbf I\to\mathbb R$. We fix $i\in\{1,2,\dots,d\}$. Let $n_i\in\mathbb N$, and let $u_{i1}, u_{i2},\dots,u_{in_i}$ be pairwise distinct elements of $I_i$.
There exists a pseudo-polynomial $W(x_1,x_2,\dots,x_d)=\sum_{k=0}^{n_i-1}A_{ik}(\mathbf x_{\{i\}'})x_i^k$ such that
$$f(x_1,x_2,\dots,x_d)=W(x_1,x_2,\dots,x_d),$$ whenever $x_i\in\{u_{i1}, u_{i2},\dots,u_{in_i}\}$.

Moreover, for each $k=0,1,\dots,n_i-1$, the function
$A_{ik}$ is a linear combina\-tion of the functions $f_{\{i\}'}^{u_{i1}}$, $f_{\{i\}'}^{u_{i2}}$, \dots, $f_{\{i\}'}^{u_{in_i}}$, where
$$
f_{\{i\}'}^{u_{ij}}(x_1,\dots,x_{i-1},x_{i+1},\dots,x_d)=f(x_1,\dots,x_{i-1},u_{ij},x_{i+1},\dots,x_d)
$$
for $j=1,2,\dots,n_i$.
\end{Lem}
\begin{proof}
We fix $i\in\{1,2,\dots,d\}$. For every $\mathbf y\in\mathbf I_{\{i\}'} $, let $ A_{ik}(\mathbf y)\in\mathbb R$ $(k=0,1,\ldots, n_i-1)$ be such that $L(x)=L_{\mathbf y}(x)=\sum_{k=0}^{n_i-1}A_{ik}(\mathbf y)x^k$ is the Lagrange interpolation polynomial of the function $f_{\{i\}}^{\mathbf y}$ for $n_i$ nodes $u_{i1}, u_{i2},\dots,u_{in_i}$. Then $L_{\mathbf y}(x)=f(y_1,\dots,y_{i-1}$, $x$, $y_{i+1},\dots,y_d)$ for $x\in\{u_{i1}, u_{i2},\dots,u_{in_i}\}$, and the coefficients $A_{ik}(\mathbf y)$ are linear combinations of the values
$f_{\{i\}}^{\mathbf y}(u_{i1})=f_{\{i\}'}^{u_{i1}}(\mathbf y),\dots,f_{\{i\}}^{\mathbf y}(u_{in_i})=f_{\{i\}'}^{u_{in_i}}(\mathbf y)$. We put $$W(x_1,x_2,\dots,x_d)=L_{\mathbf x_{\{i\}'}}(x_i)=\sum_{k=0}^{n_i-1}A_{ik}(\mathbf x_{\{i\}'})x_i^k.$$
The lemma is proved.
\end{proof}

\begin{Lem}\label{lem:pseudo2}
Let $f\colon\mathbf I\to\mathbb R$, $\mathbf n=(n_1,n_2,\dots,n_d)\in\mathbb N^d$. For each $i=1,2,\dots,d$, let $u_{i1}, u_{i2},\dots,u_{in_i}$ be pairwise distinct elements of $I_i$.
There exists a pseudo-polynomial $W$ of degree $(n_1-1,n_2-1,\dots,n_d-1)$, such that
\begin{equation}\label{eq:lem1}
    f(x_1,\dots,x_{i-1},u_{ij}, x_{i+1},\dots,x_d)=W(x_1,\dots,x_{i-1},u_{ij}, x_{i+1},\dots,x_d)
\end{equation}
for each $i=1,2,\dots,d$
and $j=1,2,\dots,n_i$.
\end{Lem}
\begin{proof}
We will define functions $g_i\colon\mathbf I\to\mathbb R$ and pseudo-polynomials $W_i\colon\mathbf I\to\mathbb R$ (for $i=1,2,\dots,d$) by induction.
First we put $g_i=f-\sum_{j=1}^{i-1}W_j$. Next, the pseudo-polynomial
$$W_i(x_1,x_2,\dots,x_d)=\sum_{k=0}^{n_i-1}A_{ik}(\mathbf x_{\{i\}'})x_i^k$$
such that $g_i(x_1,x_2,\dots,x_d)=W_i(x_1,x_2,\dots,x_d)$, whenever $x_i\in\{u_{i1},u_{i2},\dots$, $u_{in_i}\}$, is obtained using Lemma~\ref{lem:pseudo1}.

Taking the pseudo-polynomial $W=\sum_{i=1}^{d} W_i$, we obtain that equation \eqref{eq:lem1} is satisfied. The lemma is proved.
\end{proof}

\begin{Lem}\label{lem:pseudo3}
Let $\mathbf n=(n_1,n_2,\dots,n_d)\in\mathbb N^d$. If
$W\colon\mathbf I\to\mathbb R$ is a pseudo-poly\-no\-mial of degree $(n_1-1,n_2-1,\dots,n_d-1)$, then $W$ is box-$\mathbf n$-affine.
\end{Lem}

\begin{proof}
Let $W\colon\mathbf I\to\mathbb R$ be a pseudo-polynomial of degree $(n_1-1,n_2-1,\dots,n_d-1)$ of the form
$$W(\mathbf x)=\sum_{i=1}^d\sum_{k=0}^{n_i-1}A_{ik}(\mathbf x_{\{i\}'})x_i^k,$$
where $A_{ik}\colon\mathbf I_{\{i\}'}\to\mathbb R$ ($i=1,2,\dots,d$ and $k=0,1,\dots,n_i-1$) are arbitrary functions.

For a fixed $i=1,2,\dots,d$ and $k=0,1,\dots,n_i-1$ we have that $I_i\ni x_i\mapsto x_i^k$ is a box-$(n_i)$-affine function and $A_{ik}\colon\mathbf I_{\{i\}'}\to\mathbb R$. By Remark~\ref{rem:rem4}(c) we obtain that $\mathbf I\ni\mathbf x\mapsto A_{ik}(\mathbf x_{\{i\}'})x_i^k$ is box-$\mathbf n$-affine. Consequently, $W(\mathbf x)=\sum_{i=1}^d\sum_{k=0}^{n_i-1}A_{ik}(\mathbf x_{\{i\}'})x_i^k$ is box-$\mathbf n$-affine.
\end{proof}

\begin{prop}\label{prop:prop6}
Let $\mathbf n\in\mathbb N^d$ and $f,g\colon\mathbf I\to\mathbb R$
be two box-$\mathbf n$-affine functions.
Let $u_{i1}, u_{i2},\dots,u_{in_i}$ be pairwise distinct elements of $I_i$ ($i=1,2,\dots,d$).
If
\begin{equation}\label{eq:prop6}
f(x_1,\dots,x_{i-1},u_{ij}, x_{i+1},\dots,x_d)=g(x_1,\dots,x_{i-1},u_{ij}, x_{i+1},\dots,x_d)
\end{equation}
for each $i=1,2,\dots,d$
and $j=1,2,\dots,n_i$,
then $f=g$.
\end{prop}

\begin{proof}
Let $h=f-g$. For $i=1,2,\dots,d$ let $x_{i0}=x_i$ and $x_{ij}=u_{ij}$ ($j=1,2,\dots,n_i$).
Since $h$ is box-$\mathbf n$-affine, by Proposition~\ref{prop:prop5} we obtain
$$\sum_{j_1=0}^{n_1}\sum_{j_2=0}^{n_2}\dots\sum_{j_d=0}^{n_d}\frac{h(x_{1j_1},x_{2j_2},\dots,x_{dj_d})}{\prod\limits_{i=1}^d\prod\limits_{\substack{l_i=0\\l_i\neq j_i}}^{n_i}(x_{ij_i}-x_{il_i})}=0.$$
Equality~\eqref{eq:prop6} implies that $h(x_{1j_1},x_{2j_2},\dots,x_{dj_d})=0$, whenever at least one of $j_i$'s is non-zero. Consequently,
$h(x_1,x_2,\dots,x_d)=h(x_{10},x_{20},\dots,x_{d0})=0$.
\end{proof}

\begin{Thm}\label{thm:pseudo}
Let $\mathbf n=(n_1,n_2,\dots,n_d)\in\mathbb N^d$. Then a function $f\colon\mathbf I\to\mathbb R$ is box-$\mathbf n$-affine, if and only if it is a pseudo-polynomial of degree $(n_1-1,n_2-1,\dots,n_d-1)$.
\end{Thm}

\begin{proof}
For the implication ($\Leftarrow$) see Lemma~\ref{lem:pseudo3}. We need to show the implication ($\Rightarrow$).

For each $i=1,2,\dots,d$, we fix arbitrary $n_i$ pairwise distinct points $u_{i1}$, $u_{i2}$, \dots, $u_{in_i}\in I_i$.
Let $W$ be the pseudo-polynomial of degree $(n_1-1,n_2-1,\dots,n_d-1)$ obtained by Lemma~\ref{lem:pseudo2}, for which \eqref{eq:lem1} is satisfied. By Lemma~\ref{lem:pseudo3}, we have that  $W$ is box-$\mathbf n$-affine. Then, by Proposition \ref{prop:prop6} it follows that  $f=W$. The theorem is proved.
\end{proof}

\section{$\mathbf n$-regular functions}

The box-$\mathbf n$-convex functions may be very irregular (they may even be non-measurable, e.g. $f(x_1,\dots,x_d)=h(x_1)$ is box-$\mathbf n$-convex for each $d\geq2$, $\mathbf n\in\mathbb\{1,2,\dots\}^d$ and any function $h$ of one variable). However, we will use differential methods to study them. This is possible thanks to the notion of $\mathbf n$-regularity introduced in this section.

\begin{defin}\label{def:regular}
Let $f\colon\mathbf I\to\mathbb R$ and $\mathbf n\in\mathbb N^d$. We say that the function $f$ is $\mathbf n$-regular if the functions $f_{A}^{\mathbf z}\colon\mathbf I_A\to\mathbb R$ are linear combinations of box-$\mathbf n_A$-convex functions for every $A\subset\{1,2,\dots,d\}$ and $\mathbf z\in\mathbf I_{A'}$.
\end{defin}

\begin{Rem}\label{rem:rem6}
Let $\mathbf n\in\mathbb N^d$.
\begin{itemize}
\item[(a)]
The set of all $\mathbf n$-regular functions is a~linear subspace of the space of all real functions on $\mathbf I$.
\item[(b)]
If $f\colon\mathbf I\to\mathbb R$ is $\mathbf n$-regular, $A\subset\{1,2,\dots,d\}$ and $\mathbf z\in\mathbf I_{A'}$, then $f_A^{\mathbf z}$ is $\mathbf n_A$-regular.
\item[(c)]
Let $A\subset\{1,2,\dots,d\}$ and let $f\colon\mathbf I\to\mathbb R$ be the function given by the formula $f(\mathbf x)=g(\mathbf x_A)\cdot h(\mathbf x_{A'})$, where $g\colon\mathbf I_A\to\mathbb R$  and $h\colon\mathbf I_{A'}\to\mathbb R$. 
If $g$ is $\mathbf n_A$-regular and $h$ is $\mathbf n_{A'}$-regular, then $f$ is $\mathbf n$-regular (by Remark~\ref{rem:rem4}(b)).
\end{itemize}
\end{Rem}

\begin{Lem}\label{lem:regpseud}
Let $\mathbf n\in\mathbb N^d$ and let a pseudo-polynomial $W\colon\mathbf I\to\mathbb R$ of degree $\mathbf n$ be given by
$$W(x_1,x_2,\dots,x_d)=\sum_{i=1}^d\sum_{k=0}^{n_i}A_{ik}(\mathbf x_{\{i\}'})x_i^k.$$
If the functions $A_{ik}$ are $\mathbf n_{\{i\}'}$-regular for each $i=1,2,\dots,d$ and $k=0,1,\dots,n_i$, then the pseudo-polynomial $W$ is $\mathbf n$-regular.
\end{Lem}

If $d>2$, then the converse theorem does not hold. See Remark~\ref{rem:rem8}.

\begin{proof}
Let $i=1,2,\dots,d$ and $k=0,1,\dots,n_i$ be fixed. Since $I_i\ni x_i\mapsto x_i^k$ is box-$(n_i)$-convex and $A_{ik}\colon\mathbf I_{\{i\}'}\to\mathbb R$ is $\mathbf n_{\{i\}'}$-regular, we obtain (by Remark~\ref{rem:rem6}(c)) that the function $\mathbf I\ni\mathbf x\mapsto A_{ik}(\mathbf x_{\{i\}'})x_i^k$ is $\mathbf n$-regular. This implies that $W(\mathbf x)=\sum_{i=1}^d\sum_{k=0}^{n_i}A_{ik}(\mathbf x_{\{i\}'})x_i^k$ is $\mathbf n$-regular.
\end{proof}

The following lemma is a straightening of Lemma~\ref{lem:pseudo2} for $\mathbf n$-regular functions.

\begin{Lem}\label{lem:pseudo6}
Let $\mathbf n\in\mathbb N^d$ and let $f\colon\mathbf I\to\mathbb R$ be an $\mathbf n$-regular function.
For each $i=1,2,\dots,d$, let $u_{i1},u_{i2},\dots,u_{i,n_i+1}$ be pairwise distinct elements of $I_i$.
Then for each $i=1,2,\dots,d$ there exist $\mathbf n_{\{i\}'}$-regular functions $A_{i0},A_{i1},\dots$, $A_{in_i}\colon\mathbf I_{\{i\}'}\to\mathbb R$ such that the pseudo-polynomial $$
W(x_1,x_2,\dots,x_d)=\sum_{i=1}^d\sum_{k=0}^{n_i}A_{ik}(\mathbf x_{\{i\}'})x_i^k
$$
is $\mathbf n$-regular, and
\begin{equation}\label{eq:lem6}
    f(x_1,\dots,x_{i-1},u_{ij}, x_{i+1},\dots,x_d)=W(x_1,\dots,x_{i-1},u_{ij}, x_{i+1},\dots,x_d)
\end{equation}
for each $i=1,2,\dots,d$
and $j=1,2,\dots,n_i+1$.
\end{Lem}
\begin{proof} Similarly to the proof of Lemma~\ref{lem:pseudo2}, we will define by induction functions $g_i\colon\mathbf I\to\mathbb R$ and pseudo-polynomials $W_i\colon\mathbf I\to\mathbb R$ for $i=1,2,\dots,d$.
First we put $g_i=f-\sum_{j=1}^{i-1}W_j$. Next, the pseudo-polynomial
$$W_i(x_1,x_2,\dots,x_d)=\sum_{k=0}^{n_i}A_{ik}(\mathbf x_{\{i\}'})x_i^k$$
such that $g_i(x_1,x_2,\dots,x_d)=W_i(x_1,x_2,\dots,x_d)$, whenever $x_i\in\{u_{i1},u_{i2},\dots$, $u_{i,n_i+1}\}$ is obtained using Lemma~\ref{lem:pseudo1}. (Note, that we use Lemma~\ref{lem:pseudo1} with $n_i+1$ in place of $n_i$.)

By Lemma~\ref{lem:pseudo1}, Lemma~\ref{lem:regpseud} and Remark~\ref{rem:rem6}(b), we have that for each $i=1,2,\dots,d$ and $k=0,1,\dots,n_i$ the function $A_{ik}$ is $\mathbf n_{\{i\}'}$-regular as a linear combination of $\mathbf n_{\{i\}'}$-regular functions $(g_i)_{\{i\}'}^{u_{i1}}$, $(g_i)_{\{i\}'}^{u_{i2}}$, \dots, $(g_i)_{\{i\}'}^{u_{i,n_i+1}}$.

Taking the pseudo-polynomial $W=\sum_{i=1}^{d} W_i$, we obtain that
equation \eqref{eq:lem6} is satisfied. By Lemma~\ref{lem:regpseud}, $W$ is $\mathbf n$-regular. The proof is completed.
\end{proof}

\begin{prop}\label{prop:regpseud}
Let $\mathbf n\in\mathbb N^d$. A pseudo-polynomial $W\colon\mathbf I\to\mathbb R$ of degree $\mathbf n$ is $\mathbf n$-regular if and only if
for each $i=1,2,\dots,d$ there exist $\mathbf n_{\{i\}'}$-regular functions $A_{ik}\colon\mathbf I\to\mathbb R$ $(k=0,1,\dots,n_i)$ such that
$$W(x_1,x_2,\dots,x_d)=\sum_{i=1}^d\sum_{k=0}^{n_i}A_{ik}(\mathbf x_{\{i\}'})x_i^k.$$
\end{prop}
\begin{proof}
For the implication ($\Leftarrow$) see Lemma~\ref{lem:regpseud}. We need to show the implication ($\Rightarrow$). For each $i=1,2,\dots,d$, we choose arbitrary pairwise distinct $u_{i1},u_{i2},\dots$, $u_{i,n_i+1}\in I_i$. We apply Lemma~\ref{lem:pseudo6} for the function $W$ and we obtain the pseudo-polynomial
$W'(x_1,x_2,\dots,x_d)=\sum_{i=1}^d\sum_{k=0}^{n_i}A_{ik}(\mathbf x_{\{i\}'})x_i^k$ with $\mathbf n_{\{i\}'}$-regular coefficients $A_{ik}$. Since both $W$ and $W'$ are $(n_1+1,n_2+1,\dots,n_d+1)$-affine functions (by Theorem~\ref{thm:pseudo}) and they satisfy \eqref{eq:lem6}, we obtain (by Proposition~\ref{prop:prop6}) that $W=W'$.
\end{proof}

\begin{Rem}\label{rem:rem8}
According to Proposition~\ref{prop:regpseud}, a pseudo-polynomial is regular if it has a~representation with regular coefficients. However, for $d\geq3$, every regular pseudo-polynomial has also representations with non-regular coefficients. For example the $(1,1,1)$-regular pseudo-polynomial $W(x,y,z)=e^z\cdot x+e^x\cdot y+e^y\cdot z$ can be presented as $W(x,y,z)=(e^z+h(z)y)\cdot x+(e^x-h(z)x)\cdot y+e^y\cdot z$, where $h$ is arbitrary.
\end{Rem}

\begin{Thm}\label{thm:prop8}
Let $\mathbf n\in\mathbb N^d$ and let $f\colon\mathbf I\to\mathbb R$ be a box-$\mathbf n$-convex function.
For each $i=1,2,\dots,d$, let $u_{i1},u_{i2},\dots,u_{in_i}$ be pairwise distinct elements of $I_i$. If
\begin{equation}\label{eq:prop8}
f(x_1,\dots,x_{i-1},u_{ij}, x_{i+1},\dots,x_d)=0
\end{equation}
for each $i=1,2,\dots,d$
and $j=1,2,\dots,n_i$, then $f$ is $\mathbf n$-regular.
\end{Thm}

\begin{proof}
Let $A\subset\{1,2,\dots,d\}$ and $\mathbf z\in\mathbf I_{A'}$. Let $a_1<a_2<\dots<a_{|A|}$ be all elements of the set $A$ and $a'_1<a'_2<\dots<a'_{|A'|}$ be all elements of the set $A'$. We will show that the function $f_A^{\mathbf z}\colon\mathbf I_A\to\mathbb R$ is either box-$\mathbf n_A$-convex or box-$\mathbf n_A$-concave (in particular it is a linear combination of box-$\mathbf n_A$-convex functions).

Suppose that there exist $k=1,2,\dots,|A'|$, $i=a'_k$ and $j=1,2,\dots,n_i$ such that $z_k=u_{ij}$. Then, by \eqref{eq:prop8}, $f_A^{\mathbf z}=0$. Consequently, $f_A^{\mathbf z}$ is both box-$\mathbf n_A$-convex and box-$\mathbf n_A$-concave.

In the remaining case, for each $k=1,2,\dots,|A'|$ and $i=a'_k$ the numbers $z_k$, $u_{i1},u_{i2},\dots,u_{in_i}$ are pairwise distinct elements of $I_i$.
Let $\mathbf x_i=(z_k,u_{i1},$ $u_{i2},\dots$, $u_{in_i})$.
For each $i\in A$ let $x_{i0}, x_{i1},\dots,x_{in_i}$ be pairwise distinct elements of $I_i$ and $\mathbf x_i=(x_{i0}, x_{i1},\dots,x_{in_i})$.

By Remark~\ref{rem:rem1},
\begin{equation}\label{eq:eq10}
\left[\ \LARGE{\substack{\mathbf x_1\\\\\mathbf x_2\\\\\dots\\\\\mathbf x_d}};f\right]=
\left[\ \LARGE{\substack{\mathbf x_{a_1}\\\\\mathbf x_{a_2}\\\\\dots\\\\\mathbf x_{a_{|A|}}}};g\right],
\end{equation}
where $g\colon\mathbf I_A\to\mathbb R$ is given by
$$g(\mathbf x)=
\left[\LARGE{\substack{\mathbf x_{a'_1}\\\\\mathbf x_{a'_2}\\\\\dots\\\\\mathbf x_{a'_{|A'|}}}};f_{A'}^{\mathbf x}\right].$$

By \eqref{eq:prop8} and by Proposition~\ref{prop:prop5} applied to the function $f_{A'}^{\mathbf x}$ we obtain
$$g(\mathbf x)=\frac{f_{A'}^{\mathbf x}(z_1,z_2,\dots,z_{|A'|})}{\prod\limits_{k=1}^{|A'|}\prod\limits_{\substack{j=1\\\text{(for $i=a'_k$)}}}^{n_i}(z_k-u_{ij})}
=\frac{f_A^{\mathbf z}(\mathbf x)}{\prod\limits_{k=1}^{|A'|}\prod\limits_{\substack{j=1\\\text{(for $i=a'_k$)}}}^{n_i}(z_k-u_{ij})}.$$
It follows that \eqref{eq:eq10} can be written as
$$ \left[\ \LARGE{\substack{\mathbf x_1\\\\\mathbf x_2\\\\\dots\\\\\mathbf x_d}};f\right]=\frac1{\prod\limits_{k=1}^{|A'|}\prod\limits_{\substack{j=1\\\text{(for $i=a'_k$)}}}^{n_i}(z_k-u_{ij})}\cdot
\left[\ \LARGE{\substack{\mathbf x_{a_1}\\\\\mathbf x_{a_2}\\\\\dots\\\\\mathbf x_{a_{|A|}}}};f_A^{\mathbf z}\right].$$
The left side of the above equality in non-negative (because $f$ is box-$\mathbf n$-convex), which implies that the above multiple divided difference of $f_A^{\mathbf z}$ has the same sign as the denominator. Consequently, $f_A^{\mathbf z}$ is box-$\mathbf n_A$-convex if the denominator is positive, and it is box-$\mathbf n_A$-concave if the denominator is negative. The proposition is proved.
\end{proof}

By Lemma~\ref{lem:pseudo2} and Theorem~\ref{thm:prop8}, we obtain the following proposition.
\begin{prop}\label{prop:regularize}
Let $\mathbf n=(n_1,n_2,\dots,n_d)\in\mathbb N^d$ and let $f\colon\mathbf I\to\mathbb R$ be a box-$\mathbf n$-convex function. There exists a~pseudo-polynomial $W\colon\mathbf I\to\mathbb R$ of degree $(n_1-1,n_2-1,\dots,n_d-1)$ such that $f-W$ is an $\mathbf n$-regular and box-$\mathbf n$-convex function.
\end{prop}

\begin{Lem}\label{lem:dif1}
Let $\mathbf n\in\mathbb N^d$.
For every $\mathbf n$-regular function $f\colon\mathbf I\to\mathbb R$ there exist two $\mathbf n$-regular and box-$\mathbf n$-convex functions $g,h\colon\mathbf I\to\mathbb R$ such that $f=g-h$.
\end{Lem}

\begin{proof} By Definition~\ref{def:regular}, $f=\sum_{j=1}^Ja_jf_j$, where $f_1,f_2,\dots,f_J\colon\mathbf I\to\mathbb R$ are box-$\mathbf n$-convex functions.
We put $\widetilde g=\sum_{j\colon a_j\geq0}a_jf_j$ and $\widetilde h=\sum_{j\colon a_j<0}(-a_j)f_j$. Then $f=\widetilde g-\widetilde h$ and both $\widetilde g$ and $\widetilde h$ are box-$\mathbf n$-convex functions.  By Proposition~\ref{prop:regularize}, there exists box-$\mathbf n$-affine pseudo-polynomial $W$ such that $\widetilde g-W$ is $\mathbf n$-regular and box-$\mathbf n$-convex.
Let $g=\widetilde g-W$ and $h=\widetilde h-W$. Then both $g$ and $h$ are box-$\mathbf n$-convex, and $f=g-h$. Since $f$ and $g$ are $\mathbf n$-regular, we obtain that $h=g-f$ is also $\mathbf n$-regular.
\end{proof}

\begin{Lem}\label{lem:dif2}
Let $\mathbf n\in\mathbb N^d$, $i\in\{1,2,\dots,d\}$ and let $u_{i1},u_{i2},\dots,u_{in_i}\in I_i$ be pairwise distinct.
Assume that $f\colon\mathbf I\to\mathbb R$ is an~$\mathbf n$-regular function such that
$f(x_1,x_2,\dots,x_d)=0$, whenever $x_i=u_{ij}$ for $j=1,2,\dots,n_i$.
Then there exist two $\mathbf n$-regular and box-$\mathbf n$-convex functions $g,h\colon\mathbf I\to\mathbb R$ such that $f=g-h$ and $g(x_1,x_2,\dots,x_d)=h(x_1,x_2,\dots,x_d)=0$, whenever $x_i=u_{ij}$ for $j=1,2,\dots,n_i$.
\end{Lem}

\begin{proof}
The proof of Lemma~\ref{lem:dif2} is, essentially, the same as the proof of Lemma~\ref{lem:dif1}.
The only difference is that, according to Lemma~\ref{lem:pseudo2}, we can choose the pseudo-polynomial $W$ such that the following additional condition is satisfied:\\
$W(x_1,x_2,\dots,x_d)=\widetilde g(x_1,x_2,\dots,x_d)$, whenever $x_i=u_{ij}$ for $j=1,2,\dots,n_i$.
\end{proof}

\section{Integration and differentiation of $\mathbf n$-regular functions}

In this section we assume that all intervals $I_1, I_2,\dots,I_d$ are open intervals (bounded or unbounded).

\begin{Lem}\label{lem:lem15v2}
Let $i\in\{1,2,\dots,d\}$ be fixed and $\alpha_i\in I_i$. Let $\mathbf n=(n_1,n_2,\dots,n_d)\in\mathbb N^d$ be such that $n_i\geq2$ and let $f\colon\mathbf I\to\mathbb R$ be an $\mathbf n$-regular function.

Let $\psi\colon\mathbf I\to\mathbb R$ be right-derivative of $f$ with respect to the $i$th variable, i.e. $\psi$ is such that for every $\mathbf y\in\mathbf I_{\{i\}'}$ the function
$\psi_{\{i\}}^{\mathbf y}$ is the right-derivative of $f_{\{i\}}^{\mathbf y}$.
Then, the function $\psi$ is well defined. Moreover, $\psi$ is a~$(n_1,\dots,n_i-1,\dots,n_d)$-regular function, and it satisfies the equation $f(x_1,\dots,x_d)=$\\$f(x_1,\dots,\alpha_i,\dots,x_d)+\int_{\alpha_i}^{x_i} \psi(x_1,\dots,t,\dots,x_d)dt$, where $t$ and $\alpha_i$ stand at the $i$th position.

If, in addition, $f$ is box-$\mathbf n$-convex, then the function $\psi$ is box-$(n_1,\dots,n_i-1,\dots,n_d)$-convex.
\end{Lem}

\begin{proof}
For every $\mathbf y\in\mathbf I_{\{i\}'}$ the function
$f_{\{i\}}^{\mathbf y}$ is a linear combination of box-$(n_i)$-convex functions. Taking into account Remark~\ref{rem:onedim} and $n_i\geq2$, we obtain that $\psi$ is well defined and $f(\mathbf x)=f_{\{i\}}^{\mathbf y}(x_i)=f_{\{i\}}^{\mathbf y}(\alpha_i)+\int_{\alpha_i}^{x_i} \psi_{\{i\}}^{\mathbf y}(t)dt$ for $\mathbf y=\mathbf x_{\{i\}'}$.

First we show that if $f$ is $\mathbf n$-regular and box-$\mathbf n$-convex, then $\psi$ is box-$(n_1,\dots,n_i-1,\dots,n_d)$-convex.
Let $\mathbf x_k=(x_{k0},x_{k1},\dots,x_{kn_k})\in I_k^{n_k+1}$ for $k\in\{1,2,\dots,d\}\setminus\{i\}$, and $\mathbf x_i=(x_{i1},x_{i2},\dots,x_{in_i})\in I_i^{n_i}$ be vectors with pairwise distinct coordinates.
Using Remark~\ref{rem:rem1} and Lemma~\ref{lem:pochodna}, we obtain
\begin{multline*}
\left[\ \LARGE{\substack{\mathbf x_1\\\\\mathbf x_2\\\\\dots\\\\\mathbf x_d}};\psi\right]
=\left[\ \LARGE{\substack{\mathbf x_1\\\\\dots\\\\\mathbf x_{i-1}\\\\\mathbf x_{i+1}\\\\\dots\\\\\mathbf x_d}};\left[\mathbf x_i;\psi_{\{i\}}\right]\right]=\left[\ \LARGE{\substack{\mathbf x_1\\\\\dots\\\\\mathbf x_{i-1}\\\\\mathbf x_{i+1}\\\\\dots\\\\\mathbf x_d}};\left[x_{i1}, x_{i2},\ldots, x_{in_i};\psi_{\{i\}}\right]\right]\\
=\left[\ \LARGE{\substack{\mathbf x_1\\\\\dots\\\\\mathbf x_{i-1}\\\\\mathbf x_{i+1}\\\\\dots\\\\\mathbf x_d}};\sum_{j=1}^{n_i}\lim_{x_{i0}\downarrow x_{ij}}\left[x_{i0}, x_{i1},\ldots, x_{in_i};f_{\{i\}}\right]\right]
=\sum_{j=1}^{n_i}\lim_{x_{i0}\downarrow x_{ij}}\left[\LARGE{\substack{x_{10},\ x_{11},\ldots,\ x_{1n_1}\\\\x_{20},\ x_{21},\ldots,\ x_{2n_2}\\\\\dots\\\\x_{d0},\ x_{d1},\ldots,\ x_{dn_d}}};f\right]\geq0
\end{multline*}
(the last inequality follows from the box-$\mathbf n$-convexity of $f$).
We conclude that $\psi$ is box-$(n_1,\dots,n_i-1,\dots,n_d)$-convex.

Now we prove that $\psi$ is $(n_1,\dots,n_i-1,\dots,n_d)$-regular.
We put $\widetilde{\mathbf n}=(n_1,\dots,n_i-1,\dots,n_d)$. Let $A\subset\{1,2,\dots,d\}$ and $\mathbf z\in\mathbf I_{A'}$ be fixed.
We need to show that $\psi_A^{\mathbf z}$ is a linear combination of box-$\widetilde{\mathbf n}_A$-convex functions.
Let $(a_1,a_2,\dots,a_{|A|})$ be the ordered sequence of the elements of $A$ (i.e., $A=\{a_1,a_2,\dots,a_{|A|}\}$ and $a_1<a_2<\dots<a_{|A|}$). Similarly, let $(a'_1,a'_2,\dots,a'_{|A'|})$ be the ordered sequence of the elements of $A'$.
We have two cases: $i\in A$ or $i\notin A$.

First we consider the case $i\in A$. Then there exists $l$ such that $i=a_l$. By Lemma~\ref{lem:dif1}, there exist two $\mathbf n_A$-regular and box-$\mathbf n_A$-convex functions $g,h\colon\mathbf I_A\to\mathbb R$ such that $f_A^{\mathbf z}=g-h$. Let $\xi$ and $\eta$ be right-derivatives of $g$ and $h$, respectively, with respect to the $l$th variable. By the first part of the proof we obtain that $\xi$ and $\eta$ are box-$\widetilde{\mathbf n}_A$-convex functions. Consequently, $\psi_A^{\mathbf z}=\xi-\eta$ is a linear combination of box-$\widetilde{\mathbf n}_A$-convex functions.

Now we pass to the case $i\notin A$. In that case $\widetilde{\mathbf n}_A=\mathbf n_A$. Since $i\in A'$, there exists $l$ such that $i=a'_l$. We see that both $f_A^{\mathbf z}$ and $\psi_A^{\mathbf z}$ do not depend on $x_i$ (the $i$th variable of $f$ and $\psi$), which is fixed at $z_l$.
We fix pairwise distinct $u_1,u_2,\dots,u_{n_i}\in I_i$ such that $u_1=z_l$. By Lemma~\ref{lem:pseudo1}, there exists a~pseudo-polynomial $W$ of the form $W(x_1,x_2,\dots,x_d)=\sum_{k=0}^{n_i-1}A_{ik}(\mathbf x_{\{i\}'})x_i^k$ such that
$f(x_1,x_2,\dots,x_d)=W(x_1,x_2,\dots,x_d)$, whenever $x_i=u_j$ for $j=1,2,\dots,n_i$. Then the coefficients $A_{i0},A_{i1},\dots,A_{i,n_i-1}$ are
 $\mathbf n_{\{i\}'}$-regular functions.
We denote $\widetilde f=f-W$. Let $\widetilde \psi$ be a right-derivative of $\widetilde f$ with respect to the $i$th variable $x_i$. Then $\widetilde \psi=\psi-V$, where $V(\mathbf x)=\sum_{k=1}^{n_i-1}A_{ik}(\mathbf x_{\{i\}'})kx_i^{k-1}$. Obviously, the pseudo-polynomial $V$ is $\widetilde{\mathbf n}$-regular.

We consider the function $\widetilde f_{A\cup\{i\}}^{\ \widetilde{\mathbf z}}$, where $\widetilde{\mathbf z}=(z_1,\dots,z_{l-1},z_{l+1},\dots,z_{|A'|})$.
Without loss of generality (changing the order of variables, if necessary), we have
$\widetilde f_{A\cup\{i\}}^{\ \widetilde{\mathbf z}}\colon I_i\times\prod_{k\in A}I_k\to\mathbb R$. Using Lemma~\ref{lem:dif2} we obtain two $\mathbf n_{A\cup\{i\}}$-regular and box-$\mathbf n_{A\cup\{i\}}$-convex functions $g,h\colon I_i\times\prod_{k\in A}I_k\to\mathbb R$ such that $\widetilde f_{A\cup\{i\}}^{\ \widetilde{\mathbf z}}=g-h$ and $g(v_0,v_1,\dots,v_{|A|})=h(v_0,v_1,\dots,v_{|A|})=0$, whenever $v_0=u_j$ for $j=1,2,\dots,n_i$.

Let $\xi$ and $\eta$ be right-derivatives of $g$ and $h$, respectively, with respect to the variable $v_0$.
By \eqref{eq:poch},  for $v_0=u_1=z_l$ we obtain
$$\xi_{\{0\}'}^{(u_1)}(\widetilde{\mathbf v})=\xi_{\{0\}}^{\ \widetilde{\mathbf v}}(u_1)=\lim_{u_0\downarrow u_1}[u_0,u_1,u_2,\dots,u_{n_i};g_{\{0\}}^{\ \widetilde{\mathbf v}}]\cdot\prod_{j=2}^{n_i}(u_1-u_j),$$
where $\widetilde{\mathbf v}=(v_1,v_2,\dots,v_{|A|})$. Analogous identity can be obtained for $h$ and $\eta$.

For $k=1,2,\dots,|A|$, let $v_{k0},v_{k1},\dots,v_{kn_{a_k}}\in I_{a_k}$ be pairwise distinct and let $\mathbf v_k=(v_{k0},v_{k1},\dots,v_{kn_{a_k}})$.

\begin{multline*}
\left[\ \LARGE{\substack{\mathbf v_1\\\\\mathbf v_2\\\\\dots\\\\\mathbf v_{|A|}}};\xi_{\{0\}'}^{(u_1)}\right]
=\left[\ \LARGE{\substack{\mathbf v_1\\\\\mathbf v_2\\\\\dots\\\\\mathbf v_{|A|}}};\lim_{u_0\downarrow u_1}[u_0,u_1,\dots,u_{n_i};g_{\{0\}}]\cdot\prod_{j=2}^{n_i}(u_1-u_j)\right]\\
=\prod_{j=2}^{n_i}(u_1-u_j)\cdot\lim_{u_0\downarrow u_1}\left[\ \LARGE{\substack{\mathbf v_1\\\\\mathbf v_2\\\\\dots\\\\\mathbf v_{|A|}}};[u_0,u_1,\dots,u_{n_i};g_{\{0\}}]\right]
=\prod_{j=2}^{n_i}(u_1-u_j)\cdot\lim_{u_0\downarrow u_1}\left[\ \LARGE{\substack{u_0,\ u_1,\dots,\ u_{n_i}\\\\v_{10},\ v_{11},\ \dots,\ v_{1n_{a_1}}\\\\\dots\\\\v_{|A|0},\ v_{|A|1},\ \dots,\ v_{|A|n_{a_{|A|}}}}};g\right],
\end{multline*}
which has the same sign as $\prod_{j=2}^{n_i}(u_1-u_j)$ (by box-$\mathbf n_{A\cup\{i\}}$-convexity of $g$). It follows that $\xi_{\{0\}'}^{(u_1)}$ is either box-$\mathbf n_A$-convex or box-$\mathbf n_A$-concave.
Analogously, we obtain that $\eta_{\{0\}'}^{(u_1)}$ is either box-$\mathbf n_A$-convex or box-$\mathbf n_A$-concave. Therefore the function $\widetilde\psi_A^{\mathbf z}=\xi_{\{0\}'}^{(u_1)}-\eta_{\{0\}'}^{(u_1)}$ is a linear combination of box-$\mathbf n_A$-convex functions. Taking into account $\widetilde{\mathbf n}_A=\mathbf n_A$ and the equality $\psi_A^{\mathbf z}=\widetilde\psi_A^{\mathbf z}+V_A^{\mathbf z}$, we see, that
$\psi_A^{\mathbf z}$ is a linear combination of box-$\widetilde{\mathbf n}_A$-convex functions.
\end{proof}

\begin{Lem}\label{lem:calka}
Let $i\in\{1,2,\dots,d\}$ be fixed and $\alpha_i\in I_i$. Let $\mathbf n=(n_1,n_2,\dots,n_d)\in\mathbb N^d$ be such that $n_i\geq1$ and let $f\colon\mathbf I\to\mathbb R$ be an $\mathbf n$-regular function.

Let $F\colon\mathbf I\to\mathbb R$ be the function given by the formula $$
F(x_1,\dots,x_d)=\int_{\alpha_i}^{x_i} f(x_1,\dots,t,\dots,x_d)dt,
$$
where $t$ stands at the $i$th position.
Then $F$ is well defined and it is an $(n_1,\dots,n_i+1,\dots,n_d)$-regular function.

If, in addition, $f$ is box-$\mathbf n$-convex, then the function $F$ is box-$(n_1,\dots,n_i+1,\dots,n_d)$-convex.
\end{Lem}
\begin{proof}
For every $\mathbf y\in\mathbf I_{\{i\}'}$ the function
$f_{\{i\}}^{\mathbf y}$ is a linear combination of box-$(n_i)$-convex functions. Taking into account Remark~\ref{rem:onedim} and $n_i\geq1$, we obtain that $F$ is well defined.

First we show that if $f$ is $\mathbf n$-regular and box-$\mathbf n$-convex, then $F$ is box-$(n_1,\dots,n_i+1,\dots,n_d)$-convex.
Let $\mathbf x_k=(x_{k0},x_{k1},\dots,x_{kn_k})\in I_k^{n_k+1}$ for $k\in\{1,2,\dots,d\}\setminus\{i\}$, and $\mathbf x_i=(x_{i0},x_{i1},\dots,x_{i,n_i+1})\in I_i^{n_i+2}$ be vectors with pairwise distinct coordinates.
Using Remark~\ref{rem:rem1} and Lemma~\ref{lem:il.podz.}, we obtain
\begin{multline*}
\left[\ \LARGE{\substack{\mathbf x_1\\\\\mathbf x_2\\\\\dots\\\\\mathbf x_d}};F\right]
=\left[\ \LARGE{\substack{\mathbf x_1\\\\\dots\\\\\mathbf x_{i-1}\\\\\mathbf x_{i+1}\\\\\dots\\\\\mathbf x_d}};\left[\mathbf x_i;F_{\{i\}}\right]\right]=\left[\ \LARGE{\substack{\mathbf x_1\\\\\dots\\\\\mathbf x_{i-1}\\\\\mathbf x_{i+1}\\\\\dots\\\\\mathbf x_d}};\left[x_{i0}, x_{i1},\ldots, x_{i,n_i+1};F_{\{i\}}\right]\right]\\
=\left[\ \LARGE{\substack{\mathbf x_1\\\\\dots\\\\\mathbf x_{i-1}\\\\\mathbf x_{i+1}\\\\\dots\\\\\mathbf x_d}};\int_0^1t^{n_i}\left[y_{1,t}, y_{2,t},\ldots, y_{n_i+1,t};f_{\{i\}}\right]dt\right]
=\int_0^1t^{n_i}\left[\LARGE{\substack{x_{10},\ x_{11},\ldots,\ x_{1n_1}\\\\\dots\\\\y_{1,t},\ y_{2,t},\ldots,\ y_{n_i+1,t}\\\\\dots\\\\x_{d0},\ x_{d1},\ldots,\ x_{dn_d}}};f\right]dt\geq0,
\end{multline*}
where $y_{j,t}=tx_{ij}+(1-t)x_{i0}$ for $t\in[0,1]$ and $j=1,2,\dots,n_i+1$.
The~last inequality follows from the box-$\mathbf n$-convexity of $f$.
We conclude that $F$ is box-$(n_1,\dots,n_i+1,\dots,n_d)$-convex.

Now we prove that $F$ is $(n_1,\dots,n_i+1,\dots,n_d)$-regular.
We put $\widetilde{\mathbf n}=(n_1,\dots,n_i+1,\dots,n_d)$. Let $A\subset\{1,2,\dots,d\}$ and $\mathbf z\in\mathbf I_{A'}$ be fixed.
We need to show that $F_A^{\mathbf z}$ is a linear combination of box-$\widetilde{\mathbf n}_A$-convex functions.
Let $(a_1,a_2,\dots,a_{|A|})$ be the ordered sequence of the elements of $A$ (i.e., $A=\{a_1,a_2,\dots,a_{|A|}\}$ and $a_1<a_2<\dots<a_{|A|}$). Similarly, let $(a'_1,a'_2,\dots,a'_{|A'|})$ be the ordered sequence of the elements of $A'$.
We have two cases: $i\in A$ or $i\notin A$.

First we consider the case $i\in A$. Then there exists $l$ such that $i=a_l$. By Lemma~\ref{lem:dif1}, there exist two $\mathbf n_A$-regular and box-$\mathbf n_A$-convex functions $g,h\colon\mathbf I_A\to\mathbb R$ such that $f_A^{\mathbf z}=g-h$. Let $G(y_1,\dots,y_{|A|})=\int_{\alpha_i}^{y_l} g(y_1,\dots,t,\dots,y_{|A|})dt$ and $H(y_1,\dots,y_{|A|})=\int_{\alpha_i}^{y_l} h(y_1,\dots,t,\dots,y_{|A|})dt$, where $t$ stands at the $l$th position.
By the first part of the proof we obtain that $G$ and $H$ are box-$\widetilde{\mathbf n}_A$-convex functions. Consequently, $F_A^{\mathbf z}=G-H$ is a linear combination of box-$\widetilde{\mathbf n}_A$-convex functions.

Now we pass to the case $i\notin A$. In that case $\widetilde{\mathbf n}_A=\mathbf n_A$. Since $i\in A'$, there exists $l$ such that $i=a'_l$. We see that both $f_A^{\mathbf z}$ and $F_A^{\mathbf z}$ do not depend on $x_i$ (the $i$th variable of $f$ and $F$), which is fixed at $z_l$.
We fix pairwise distinct $u_1,u_2,\dots,u_{n_i}\in I_i$ such that $u_j<\min(\alpha_i,z_l)$ for $j=1,2,\dots,n_i$. By Lemma~\ref{lem:pseudo1}, there exists a~pseudo-polynomial $W$ of the form $W(x_1,x_2,\dots,x_d)=\sum_{k=0}^{n_i-1}A_{ik}(\mathbf x_{\{i\}'})x_i^k$ such that
$f(x_1,x_2,\dots,x_d)=W(x_1,x_2,\dots,x_d)$, whenever $x_i=u_j$ for $j=1,2,\dots,n_i$. Then the coefficients $A_{i0},A_{i1},\dots,A_{i,n_i-1}$ are
 $\mathbf n_{\{i\}'}$-regular functions.
We denote $\widetilde f=f-W$. Let 
$$
\widetilde F(x_1,\dots,x_d)=\int_{\alpha_i}^{x_i}\widetilde f(x_1,\dots,t,\dots,x_d)dt,
$$ where $t$ stands at the $i$th position. Then $\widetilde F=F-V$, where
$$
V(\mathbf x)=\sum_{k=0}^{n_i-1}A_{ik}(\mathbf x_{\{i\}'})\frac{x_i^{k+1}-\alpha_i^{k+1}}{k+1}.
$$
Obviously, the pseudo-polynomial $V$ is $\widetilde{\mathbf n}$-regular.

We consider the function $\widetilde f_{A\cup\{i\}}^{\ \widetilde{\mathbf z}}$, where $\widetilde{\mathbf z}=(z_1,\dots,z_{l-1},z_{l+1},\dots,z_{|A'|})$.
Without loss of generality (changing the order of variables, if necessary), we have
$\widetilde f_{A\cup\{i\}}^{\ \widetilde{\mathbf z}}\colon I_i\times\prod_{k\in A}I_k\to\mathbb R$. Using Lemma~\ref{lem:dif2} we obtain two $\mathbf n_{A\cup\{i\}}$-regular and box-$\mathbf n_{A\cup\{i\}}$-convex functions $g,h\colon I_i\times\prod_{k\in A}I_k\to\mathbb R$ such that $\widetilde f_{A\cup\{i\}}^{\ \widetilde{\mathbf z}}=g-h$ and $g(v_0,v_1,\dots,v_{|A|})=h(v_0,v_1,\dots,v_{|A|})=0$, whenever $v_0=u_j$ for $j=1,2,\dots,n_i$.
Let $G(y_0,y_1,\dots,y_{|A|})=\int_{\alpha_i}^{y_0} g(t,y_1,\dots,y_{|A|})dt$ and $H(y_0,y_1,\dots,y_{|A|})=\int_{\alpha_i}^{y_0} h(t,y_1,\dots,y_{|A|})dt$.

For $k=1,2,\dots,|A|$, let $v_{k0},v_{k1},\dots,v_{kn_{a_k}}\in I_{a_k}$ be pairwise distinct and let $\mathbf v_k=(v_{k0},v_{k1},\dots,v_{kn_{a_k}})$.
By Remark~\ref{rem:rem1} and \eqref{eq:expanded}, for every $t$ between $\alpha_i$ and $z_l$ we have
\begin{multline*}
\left[\ \LARGE{\substack{t,\ u_1,\dots,\ u_{n_i}\\\\v_{10},\ v_{11},\ \dots,\ v_{1n_{a_1}}\\\\\dots\\\\v_{|A|0},\ v_{|A|1},\ \dots,\ v_{|A|n_{a_{|A|}}}}};g\right]
=\left[\ \LARGE{\substack{\mathbf v_1\\\\\mathbf v_2\\\\\dots\\\\\mathbf v_{|A|}}};[t,u_1,\dots,u_{n_i};g_{\{0\}}]\right]\\
=\left[\ \LARGE{\substack{\mathbf v_1\\\\\mathbf v_2\\\\\dots\\\\\mathbf v_{|A|}}};\frac{g_{\{0\}'}^{(t)}}{\prod_{j=1}^{n_i}(t-u_j)}\right]
=\frac1{\prod_{j=1}^{n_i}(t-u_j)}\left[\ \LARGE{\substack{\mathbf v_1\\\\\mathbf v_2\\\\\dots\\\\\mathbf v_{|A|}}};g_{\{0\}'}^{(t)}\right].
\end{multline*}
Since $t-u_j>t-\min(\alpha_i,z_l)\geq0$ for each $j$, and $g$ is box-$\mathbf n_{A\cup\{i\}}$-convex, we obtain that $\left[\ \LARGE{\substack{\mathbf v_1\\\\\mathbf v_2\\\\\dots\\\\\mathbf v_{|A|}}};g_{\{0\}'}^{(t)}\right]\geq0$. It follows that
$$\left[\ \LARGE{\substack{\mathbf v_1\\\\\mathbf v_2\\\\\dots\\\\\mathbf v_{|A|}}};G_{\{0\}'}^{(z_l)}\right]
=\left[\ \LARGE{\substack{\mathbf v_1\\\\\mathbf v_2\\\\\dots\\\\\mathbf v_{|A|}}};\int_{\alpha_i}^{z_l}g_{\{0\}'}^{(t)}dt\right]
=\int_{\alpha_i}^{z_l}\left[\ \LARGE{\substack{\mathbf v_1\\\\\mathbf v_2\\\\\dots\\\\\mathbf v_{|A|}}};g_{\{0\}'}^{(t)}\right]dt$$
is either non-negative (if $z_l\geq\alpha_i$) or non-positive (if $z_l<\alpha_i$). Consequently, $G_{\{0\}'}^{(z_l)}$ is box-$\mathbf n_A$-convex or box-$\mathbf n_A$-concave.
Similarly, $H_{\{0\}'}^{(z_l)}$ is box-$\mathbf n_A$-convex or box-$\mathbf n_A$-concave. Therefore the function $\widetilde F_A^{\mathbf z}=G_{\{0\}'}^{(z_l)}-H_{\{0\}'}^{(z_l)}$ is a linear combination of box-$\mathbf n_A$-convex functions. Taking into account $\widetilde{\mathbf n}_A=\mathbf n_A$ and the equality $F_A^{\mathbf z}=\widetilde F_A^{\mathbf z}+V_A^{\mathbf z}$, we see, that
$F_A^{\mathbf z}$ is a linear combination of box-$\widetilde{\mathbf n}_A$-convex functions.
\end{proof}

\section{Integral representation of box-monotone functions}

In this section, we study box-$(1,1,\dots,1)$-convex functions.
We will call them \emph{box-monotone} functions.
Our aim is to present the integral representation of box-monotone functions. The general idea is based on the observation that box-monotone functions behave similarly to cumulative distribution functions of probability distributions. However, while cumulative distribution functions are always bounded and right-continuous (or left-continuous), box-monotone functions, in general, do not have these properties.

We start from some results concerning real functions of one variable.

Let $I\subset\mathbb R$ be an open interval and $f\colon I\to\mathbb R$ be a~function with locally finite variation (which is equivalent to $f$ being a~difference of two non-decreasing functions). For $x\in I$ we denote one-side limits of $f$ at $x$ as $f(x-)=\lim_{u\uparrow x}f(u)$ and $f(x+)=\lim_{u\downarrow x}f(u)$.
We say that  $f$  is a~\emph{jump function} if for every $x,y\in I$, $x<y$, we have
\begin{equation}\label{eq:jump}
f(y-)-f(x+)=\sum_{x<t<y}(f(t+)-f(t-)).
\end{equation}
Note that our assumption that the function $f$ has locally finite variation implies that $f(t+)-f(t-)=0$ for all but countably many $t\in(x,y)$, and $\sum_{x<t<y}|f(t+)-f(t-)|<\infty$.

In the following lemma, we obtain a~decomposition of a~function with locally finite variation, which is a~counterpart of the well known Lebesgue decom\-po\-sition of right-continuous functions with finite variation.

\begin{Lem}\label{lem:decomp}
Let $I\subset\mathbb R$ be an open interval and let $\alpha\in I$ be fixed. For a~function $f\colon I\to\mathbb R$ with locally finite variation, there exist unique functions $f_L,f_R,f_c\colon I\to\mathbb R$ such that $f_L$ and $f_R$ are jump functions, $f_L(\alpha)=f_R(\alpha)=0$, $f_L$ is a~left-continuous function, $f_R$ is a~right-continuous function, $f_c$ is a~continuous function, and $f=f_L+f_R+f_c$.

Moreover, if the function $f$ is non-decreasing, then the functions $f_L$, $f_R$ and $f_c$ are also non-decreasing.
\end{Lem}

\begin{proof}
First we show that whenever the functions $f_L$, $f_R$ and $f_c$ satisfying the conditions given in the lemma exist, then they are unique.
Assume that for some $f$ there are two such triplets: $f_L$, $f_R$, $f_c$, and $\widetilde f_L$, $\widetilde f_R$, $\widetilde f_c$. Then the functions $g_L=f_L-\widetilde f_L$, $g_R=f_R-\widetilde f_R$ and $g_c=f_c-\widetilde f_c$ form a triplet for the function $g=f-f=0$. It is enough to show that $g_L=g_R=g_c=0$.
Since $g_L=g-g_R-g_c=-g_R-g_c$, we obtain that $g_L$ is not only left-continuous, but also right-continuous, hence it is continuous. Let $x,y\in I$, $x<y$. Since $g_L$ is a~continuous jump function, we have
$$g_L(y)-g_L(x)=g_L(y-)-g_L(x+)=\sum_{x<t<y}(g_L(t+)-g_L(t-))=0.$$
Therefore $g_L$ is constant. Taking into account the condition $g_L(\alpha)=0$, we get $g_L=0$. Similarly, $g_R=0$. Consequently, $g_c=g-g_L-g_R=0$.

We will show that the following functions $f_L$, $f_R$ and $f_c$ satisfy the conditions given in the lemma.
\begin{align*}
f_L(x)&=\sum_{\alpha\leq t<x}(f(t+)-f(t))-\sum_{x\leq t<\alpha}(f(t+)-f(t)),\\
f_R(x)&=\sum_{\alpha<t\leq x}(f(t)-f(t-))-\sum_{x<t\leq\alpha}(f(t)-f(t-)),\\
f_c(x)&=f(x)-f_L(x)-f_R(x).
\end{align*}
The above functions are well defined because the~function $f$ has locally finite variation. Clearly, $f_L(\alpha)=f_R(\alpha)=0$ and $f=f_L+f_R+f_c$.

For $x,y\in I$, $x<y$ we have 
\begin{equation}\label{eq:L1}
f_L(y)-f_L(x)=\sum_{x\leq t<y}(f(t+)-f(t))\text{ and }
f_R(y)-f_R(x)=\sum_{x<t\leq y}(f(t)-f(t-)).
\end{equation}
It follows that for every $x\in I$ we have
$$f_L(x)-f_L(x-)=\lim_{u\uparrow x}(f_L(x)-f_L(u))=\lim_{u\uparrow x}\sum_{u\leq t<x}(f(t+)-f(t))=0,$$
$$f_L(x+)-f_L(x)=\lim_{u\downarrow x}(f_L(u)-f_L(x))=\lim_{u\downarrow x}\sum_{x\leq t<u}(f(t+)-f(t))=f(x+)-f(x).$$
Similarly, for every $x\in I$ we have $f_R(x+)-f_R(x)=0$ and $f_R(x)-f_R(x-)=f(x)-f(x-)$.
It follows that $f_L$ is left-continuous and $f_R$ is right-continuous. Moreover, for every $x\in I$ we obtain
$$f_c(x+)-f_c(x)=(f(x+)-f(x))-(f_L(x+)-f_L(x))-(f_R(x+)-f_R(x))=0,$$
$$f_c(x)-f_c(x-)=(f(x)-f(x-))-(f_L(x)-f_L(x-))-(f_R(x)-f_R(x-))=0,$$
hence $f_c$ is continuous.

Now we show that $f_L$ and $f_R$ are jump functions. Let $x,y\in I$, $x<y$. For every $t\in(x,y)$ we have
$$f_L(t+)-f_L(t-)=f_L(t+)-f_L(t)=f(t+)-f(t).$$
By \eqref{eq:L1} and left-continuity of $f_L$, we obtain
\begin{equation*}
\begin{split}
\sum_{x<t<y}(f_L(t+)-f_L(t-))=&\sum_{x<t<y}(f(t+)-f(t))=(f_L(y)-f_L(x))-(f(x+)-f(x))\\=&(f_L(y)-f_L(x))-(f_L(x+)-f_L(x))
=f_L(y-)-f_L(x+),
\end{split}
\end{equation*}
hence $f_L$ is a jump function.
Similarly, we obtain that $f_R$ is a jump function.

Now we assume that $f$ is non-decreasing. By \eqref{eq:L1}, we immediately obtain that the functions $f_L$ and $f_R$ are non-decreasing. We will prove that $f_c$ is also non-decreasing. Aiming at the contradiction, we assume that there exist $x,y\in I$ such that $x<y$ and $f_c(x)>f_c(y)$.
The inequality $f_c(x)>f_c(y)$ can be written as
$$(f_L(y)-f_L(x))+(f_R(y)-f_R(x))>f(y)-f(x).$$
By \eqref{eq:L1}, we get
$$\sum_{x<t<y}(f(t+)-f(t-))>f(y-)-f(x+).$$
Then there exist $x<t_1<t_2<\dots<t_n<y$ such that
$$\sum_{k=1}^n(f(t_k+)-f(t_k-))>f(y-)-f(x+),$$
and, consequently, there exist $x<u_0<t_1<u_1<t_2<\dots<t_n<u_n<y$ such that
$$\sum_{k=1}^n(f(u_k)-f(u_{k-1}))>\sum_{k=1}^n(f(t_k+)-f(t_k-))>f(u_n)-f(u_0).$$
We obtained a~contradiction, which implies that $f_c$ is non-decreasing.
\end{proof}

In the sequel we may continue to deal with the decomposition $f=f_L+f_R+f_c$ given by Lemma~\ref{lem:decomp}. However, for simplicity, we will use the decomposition of $f$ to two functions, given by the following corollary.

\begin{Cor}\label{cor:decomp}
Let $I\subset\mathbb R$ be an open interval and let $\alpha\in I$ be fixed. For a~function $f\colon I\to\mathbb R$ with locally finite variation, there exist unique functions $f_L,f_r\colon I\to\mathbb R$ such that $f_L$ is a left-continuous jump function and $f_r$ is a~right-continuous function such that $f_L(\alpha)=0$ and $f=f_L+f_r$.

Moreover, if the function $f$ is non-decreasing, then the functions $f_L$ and $f_r$ are also non-decreasing.
\end{Cor}
\begin{proof}
Let $f_L$, $f_R$ and $f_c$ be given by Lemma~\ref{lem:decomp}. We put $f_r=f_R+f_c$. The uniqueness can be proven similarly as in the proof of Lemma~\ref{lem:decomp}.
\end{proof}

\begin{Rem}\label{rem:Lr}
Note that the operations $\cdot_L$, $\cdot_R$, $\cdot_c$ and $\cdot_r$ are linear operators. To see it, one can use the explicit formulas presented in Corollary~\ref{cor:decomp} and the proof of~Lemma~\ref{lem:decomp}.
Moreover, we have $(f_L)_L=f_L$, $(f_R)_R=f_R$, $(f_c)_c=f_c$ and $(f_r)_r=f_r$. It follows that all the considered operators are linear projections in the space of functions with locally finite variation. In particular,
$(f_L)_r=f_L-(f_L)_L=0$. Similarly, $(f_r)_L=f_r-(f_r)_r=0$.
\end{Rem}

We are ready to investigate box-monotone functions of $d$ variables.
First we need to extend the decomposition presented in Corollary~\ref{cor:decomp} to functions of many variables.
If a~function $f\colon\mathbf I\to\mathbb R$ is $(1,1,\dots,1)$-regular, then for every $i\in\{1,2,\dots,d\}$ and $\mathbf y\in\mathbf I_{\{i\}'}$ the function $f_{\{i\}}^{\mathbf y}$ is the function of one variable with locally finite variation. It follows that the following definition is valid.

\begin{defin}\label{def:def6}
Let $f\colon\mathbf I\to\mathbb R$ be a~$(1,1,\dots,1)$-regular function. For $i=1,2,\dots,$ $d$ and $\alpha_i\in I_i$, we define the functions $f_{(i,L)},f_{(i,r)}\colon\mathbf I\to\mathbb R$ as follows: for every $\mathbf y\in\mathbf I_{\{i\}'}$ we put $(f_{(i,L)})_{\{i\}}^{\mathbf y}=(f_{\{i\}}^{\mathbf y})_L$ and $(f_{(i,r)})_{\{i\}}^{\mathbf y}=(f_{\{i\}}^{\mathbf y})_r$.
\end{defin}

\begin{Rem}\label{rem:Lr2}
Intuitively, in Definition~\ref{def:def6} we apply operation $\cdot_L$ or $\cdot_r$ to the $i$th variable of the function $f$.
Note that $f=f_{(i,L)}+f_{(i,r)}$ and for every $\mathbf y\in\mathbf I_{\{i\}'}$ we have that $(f_{(i,L)})_{\{i\}}^{\mathbf y}$ is a left-continuous jump function, $(f_{(i,r)})_{\{i\}}^{\mathbf y}$ is a~right-continuous function, and $(f_{(i,L)})_{\{i\}}^{\mathbf y}(\alpha_i)=0$. According to Corollary~\ref{cor:decomp}, $f_{(i,L)}$ and $f_{(i,r)}$ are the unique functions satisfying the above conditions.
Moreover, by Remark~\ref{rem:Lr}, we have that $\cdot_{(i,L)}$ and $\cdot_{(i,r)}$ are linear operators satisfying $(f_{(i,L)})_{(i,L)}=f_{(i,L)}$, $(f_{(i,r)})_{(i,r)}=f_{(i,r)}$ and $(f_{(i,L)})_{(i,r)}=(f_{(i,r)})_{(i,L)}=0$.
\end{Rem}

\begin{Lem}\label{lem:boxmon}
Let $f\colon\mathbf I\to\mathbb R$ be a~$(1,1,\dots,1)$-regular function. For every $i\in\{1,2,\dots,d\}$ and $\alpha_i\in I_i$ the functions $f_{(i,L)},f_{(i,r)}\colon\mathbf I\to\mathbb R$ are $(1,1,\dots,1)$-regular.

If, in addition, $f$ is box-monotone, then the functions $f_{(i,L)}$ and $f_{(i,r)}$ are also box-monotone.
\end{Lem}
\begin{proof}
First, we prove that if $f$ is $(1,1,\dots,1)$-regular and box-monotone, then the functions $f_{(i,L)}$ and $f_{(i,r)}$ are box-monotone. Let $g$ be any of the functions $f$, $f_{(i,L)}$ and $f_{(i,r)}$. The function $g$ is box-monotone, if for every vectors $(x_{10},x_{11})$ $\in I_1^2$, \dots, $(x_{d0},x_{d1})\in I_d^2$ with pairwise distinct coordinates, we have
\begin{equation*}
\left[\ \LARGE{\substack{x_{10},\ x_{11}\\\\x_{20},\ x_{21}\\\\\dots\\\\x_{d0},\ x_{d1}}};g\right]=[x_{i0},x_{i1};\widetilde g ]\geq0,
\end{equation*}
where the function $\widetilde g\colon I_i\to\mathbb R$ is given by
$$\widetilde g(x_i):=\left[\LARGE{\substack{x_{10},\ x_{11}\\\\\dots\\\\x_{i-1,0},\ x_{i-1,1}\\\\x_{i+1,0},\ x_{i+1,1}\\\\\dots\\\\x_{d0},\ x_{d1}}};g_{\{i\}'}^{(x_i)}\right].$$
The above inequality implies that the function $g$ is box-monotone if and only if 
$\widetilde g$ is non-decreasing (for every $(x_{10},x_{11})$, \dots, $(x_{i-1,0},x_{i-1,1})$, $(x_{i+1,0},x_{i+1,1})$, \dots, $(x_{d0},x_{d1})$).
By the expanded form of $\widetilde g$ given by Proposition~\ref{prop:prop5}, the function $\widetilde g$ is a linear combination of functions of the form $g_{\{i\}}^{\mathbf y}$ (with the coefficients and $\mathbf y$'s depending on $(x_{10},x_{11})$, \dots, $(x_{i-1,0},x_{i-1,1})$, $(x_{i+1,0},x_{i+1,1})$, \dots, $(x_{d0},x_{d1})$ only). Thus, by the linearity of $\cdot_L$, $\cdot_r$, $\cdot_{(i,L)}$ and $\cdot_{(i,r)}$, we obtain that $\widetilde{f_{(i,L)}}=(\widetilde f)_L$ and $\widetilde{f_{(i,r)}}=(\widetilde f)_r$.
Using Corollary~\ref{cor:decomp} and the fact that $\widetilde f$ is non-decreasing, we obtain that the functions $\widetilde{f_{(i,L)}}$ and $\widetilde{f_{(i,r)}}$ are non-decreasing. We end up with the conclusion that the functions $f_{(i,L)}$ and $f_{(i,r)}$ are box-monotone.

Now, we prove the $(1,1,\dots,1)$-regularity of $f_{(i,L)}$ and $f_{(i,r)}$.
Since $f_{(i,r)}=f-f_{(i,L)}$, it is enough to show the $(1,1,\dots,1)$-regularity of $f_{(i,L)}$.
Let $A\subset\{1,2,\dots,d\}$ and $\mathbf z\in\mathbf I_{A'}$ be fixed.
We need to show that $(f_{(i,L)})_A^{\mathbf z}$ is a linear combination of box-monotone functions.
Let $(a_1,a_2,\dots,a_{|A|})$ be the ordered sequence of the elements of $A$ (i.e., $A=\{a_1,a_2,\dots,a_{|A|}\}$ and $a_1<a_2<\dots<a_{|A|}$). Similarly, let $(a'_1,a'_2,\dots,a'_{|A'|})$ be the ordered sequence of the elements of $A'$.
We have two cases: $i\in A$ or $i\notin A$.

First we consider the case $i\in A$. Then there exists $l$ such that $i=a_l$. By Lemma~\ref{lem:dif1}, there exist two $(1,1,\dots,1)$-regular and box-monotone functions $g,h\colon\mathbf I_A\to\mathbb R$ such that $f_A^{\mathbf z}=g-h$.
By the first part of the proof, we obtain that $g_{(l,L)}$ and $h_{(l,L)}$ are box-monotone functions. Consequently, $(f_{(i,L)})_A^{\mathbf z}=g_{(l,L)}-h_{(l,L)}$ is a linear combination of box-monotone functions.

Now we pass to the case $i\notin A$. Since $i\in A'$, there exists $l$ such that $i=a'_l$. We see that both $f_A^{\mathbf z}$ and $(f_{(i,L)})_A^{\mathbf z}$ do not depend on $x_i$ (the $i$th variable of $f$ and $f_{(i,L)}$), which is fixed at $z_l$. If $z_l=\alpha_i$, then $(f_{(i,L)})_A^{\mathbf z}=0$ is a~box-monotone function. In the sequel, we assume that $z_l\neq\alpha_i$.

We consider the function $f_{A\cup\{i\}}^{\ \widetilde{\mathbf z}}$, where $\widetilde{\mathbf z}=(z_1,\dots,z_{l-1},z_{l+1},\dots,z_{|A'|})$.
Without loss of generality (changing the order of variables, if necessary), we have
$f_{A\cup\{i\}}^{\ \widetilde{\mathbf z}}\colon I_i\times\prod_{k\in A}I_k\to\mathbb R$. Using Lemma~\ref{lem:dif2}, we obtain two $(1,1,\dots,1)$-regular and box-monotone functions $g,h\colon I_i\times\prod_{k\in A}I_k\to\mathbb R$ such that $f_{A\cup\{i\}}^{\ \widetilde{\mathbf z}}=g-h$ and we have $g(v_0,v_1,\dots,v_{|A|})=h(v_0,v_1,\dots,v_{|A|})=0$ whenever $v_0=\alpha_i$ (we number the variables of $g$ and $h$ starting from 0).

For $k=1,2,\dots,|A|$, let $v_{k0},v_{k1}\in I_{a_k}$ be pairwise distinct.
By the first part of the proof, we obtain that $g_{(0,L)}$ is the box-monotone function. Consequently,
\begin{equation*}
\begin{split}
0\leq&\left[\ \LARGE{\substack{z_l,\ \alpha_i\\\\v_{10},\ v_{11}\\\\\dots\\\\v_{|A|0},\ v_{|A|1}}};g_{(0,L)}\right]
=\frac{
\left[\ \LARGE{\substack{v_{10},\ v_{11}\\\\\dots\\\\v_{|A|0},\ v_{|A|1}}};(g_{(0,L)})_{\{0\}'}^{(z_l)}\right]
-\left[\ \LARGE{\substack{v_{10},\ v_{11}\\\\\dots\\\\v_{|A|0},\ v_{|A|1}}};(g_{(0,L)})_{\{0\}'}^{(\alpha_i)}\right]}{z_l-\alpha_i}\\\\
=&\frac1{z_l-\alpha_i}\cdot\left[\ \LARGE{\substack{v_{10},\ v_{11}\\\\\dots\\\\v_{|A|0},\ v_{|A|1}}};(g_{(0,L)})_{\{0\}'}^{(z_l)}\right].
\end{split}
\end{equation*}
It follows that the sign of the last divided difference depends only on the sign of $z_l-\alpha_i$. Consequently, $(g_{(0,L)})_{\{0\}'}^{(z_l)}$ is either box-$(1,1,\dots,1)$-convex or box-$(1,1,\dots,1)$-concave. Similarly, we can obtain that $(h_{(0,L)})_{\{0\}'}^{(z_l)}$ is either box-$(1,1,\dots,1)$-convex or box-$(1,1,\dots,1)$-concave.
We end up with the conclusion that $(f_{(i,L)})_A^{\mathbf z}=(g_{(0,L)})_{\{0\}'}^{(z_l)}-(h_{(0,L)})_{\{0\}'}^{(z_l)}$ is a linear combination of box-$(1,1,\dots,1)$-convex functions.
The proof is finished.
\end{proof}

\begin{Lem}\label{lem:16}
Let $f\colon\mathbf I\to\mathbb R$ be a~$(1,1,\dots,1)$-regular function and let $\boldsymbol\alpha=(\alpha_1,\alpha_2,\dots,\alpha_d)\in\mathbf I$ be fixed. If $i,j\in\{1,2,\dots,d\}$ and $a,b\in\{L,r\}$, then $f_{(i,a)(j,b)}=f_{(j,b)(i,a)}$.
\end{Lem}
By $f_{(i,a)(j,b)}$ we mean $(f_{(i,a)})_{(j,b)}$. The same convention applies in other cases.

\begin{proof}
If $i=j$, then the lemma follows from Remark~\ref{rem:Lr2}.

First we show that the conclusion of the lemma holds for $a=b=L$. It is enough to prove that for every $\mathbf y\in\mathbf I_{\{i,j\}'}$ and for $\widetilde f=f_{\{i,j\}}^{\mathbf y}$ we have $\widetilde f_{(1,L)(2,L)}=\widetilde f_{(2,L)(1,L)}$. By Lemma~\ref{lem:dif1}, there exist two $(1,1)$-regular and box-monotone functions $g,h\colon I_i\times I_j\to\mathbb R$ such that $\widetilde f=g-h$.

Since $g$ is box-monotone, by Lemma~\ref{lem:boxmon}, we obtain that both functions
$g_{(1,L)(2,L)(1,r)}$ and $g_{(1,L)(2,r)(1,r)}$
are box-monotone. Therefore
$$g_{(1,L)(2,L)(1,r)}=(g_{(1,L)(2,L)}+g_{(1,L)(2,r)})_{(1,r)}-g_{(1,L)(2,r)(1,r)}
=g_{(1,L)(1,r)}-g_{(1,L)(2,r)(1,r)}=-g_{(1,L)(2,r)(1,r)}$$
is a box-$(1,1)$-affine function. Similarly (exchanging all $\cdot_{(1,L)}$ and $\cdot_{(1,r)}$), we obtain that $g_{(1,r)(2,L)(1,L)}$ is box-$(1,1)$-affine.
It follows that
\begin{equation*}
\begin{split}
g_{(2,L)(1,L)}-g_{(1,L)(2,L)}
=&(g_{(1,L)}+g_{(1,r)})_{(2,L)(1,L)}-(g_{(1,L)(2,L)(1,L)}+g_{(1,L)(2,L)(1,r)})\\
=&g_{(1,r)(2,L)(1,L)}-g_{(1,L)(2,L)(1,r)}
\end{split}
\end{equation*}
is also box-$(1,1)$-affine. On the other hand, we have $g_{(2,L)(1,L)}(v_1,v_2)=0=g_{(1,L)(2,L)}(v_1,v_2)$, whenever $v_1=\alpha_i$ or $v_2=\alpha_j$.
By Proposition~\ref{prop:prop6}, we get $g_{(2,L)(1,L)}=g_{(1,L)(2,L)}$.
Similarly, we obtain $h_{(2,L)(1,L)}=h_{(1,L)(2,L)}$. Con\-se\-quent\-ly, $\widetilde f_{(2,L)(1,L)}=\widetilde f_{(1,L)(2,L)}$ and
$f_{(i,L)(j,L)}=f_{(j,L)(i,L)}$.

We proved one of the four required identities. We use it to prove the other three:
$$f_{(i,L)(j,r)}=f_{(i,L)}-f_{(i,L)(j,L)}=f_{(i,L)}-f_{(j,L)(i,L)}=f_{(j,r)(i,L)},$$
$$f_{(i,r)(j,L)}=f_{(j,L)}-f_{(i,L)(j,L)}=f_{(j,L)}-f_{(j,L)(i,L)}=f_{(j,L)(i,r)},$$
$$f_{(i,r)(j,r)}=f_{(j,r)}-f_{(i,L)(j,r)}=f_{(j,r)}-f_{(j,r)(i,L)}=f_{(j,r)(i,r)}.$$
This completes the proof of the lemma.
\end{proof}

In the following theorem we study a $d$-dimensional counterpart of the decomposition introduced in Definition~\ref{def:def6}.

\begin{Thm}\label{th:rozklad}
Let $f\colon\mathbf I\to\mathbb R$ be a~$(1,1,\dots,1)$-regular function and let $\boldsymbol\alpha=(\alpha_1,\alpha_2,\dots,\alpha_d)\in\mathbf I$ be fixed. There exist unique $(1,1,\dots,1)$-regular functions $f_\mathbf b$, $\mathbf b=(b_1,b_2,\dots,b_d)\in\{L,r\}^d$, such that
$$f=\sum_{\mathbf b\in\{L,r\}^d}f_{\mathbf b},$$
and for every $\mathbf b\in\{L,r\}^d$, $i=1,2,\dots,d$ and $\mathbf y\in\mathbf I_{\{i\}'}$, we have:
\begin{itemize}
\item[(a)] if $b_i=L$, then $(f_\mathbf b)_{\{i\}}^{\mathbf y}$ is a~left-continuous jump function, $(f_\mathbf b)_{\{i\}}^{\mathbf y}(\alpha_i)=0$,
\item[(b)] if $b_i=r$, then $(f_\mathbf b)_{\{i\}}^{\mathbf y}$ is a~right-continuous function.
\end{itemize}

If, in addition, $f$ is box-monotone, then the functions $f_{\mathbf b}$ are also box-monotone.
\end{Thm}

\begin{proof}
Let $\mathbf b=(b_1,b_2,\dots,b_d)\in\{L,r\}^d$. We consider the functions $f_{(b_1,b_2,\dots,b_k)} $, $ k=1, \ldots, d$, defined inductively as follows: $f_{(b_1)}= f_{(1,b_1)}$ and $f_{(b_1,b_2,\dots,b_k)}= (f_{(b_1,b_2,\dots,b_{k-1})})_{(k,b_k)}$, $ k=2, \ldots, d$. We will check that $f_\mathbf b=f_{(b_1,b_2,\dots,b_d)}$ satisfies the required conditions.

By Lemma~\ref{lem:boxmon}, the functions $f_{\mathbf b}$ are $(1,1,\dots,1)$-regular. If, in addition, $f$ is box-monotone, then the functions $f_{\mathbf b}$ are also box-monotone. By Remark~\ref{rem:Lr2}, we have
$$f=\sum_{b_1\in\{L,r\}}f_{(b_1)}=\sum_{b_1,b_2\in\{L,r\}}f_{(b_1,b_2)}=\dots=\sum_{\mathbf b\in\{L,r\}^d}f_{\mathbf b}.$$

By Lemma~\ref{lem:16}, and the last part of Remark~\ref{rem:Lr2}, we also have $f_\mathbf b=(f_\mathbf b)_{(i,b_i)}$. Consequently, conditions (a) and (b) of the theorem are satisfied.

It remains to prove the uniqueness of the functions $f_\mathbf b$. We will show by induction on $k=0,1,\dots,d$, that for every
for $a_1,a_2,\dots,a_k\in\{L,r\}$ the function 
\begin{equation}\label{eq:okropne}
\sum_{\substack{\mathbf b\in\{L,r\}^d\\(b_1,\dots,b_k)=(a_1,\dots,a_k)}}f_{\mathbf b}
\end{equation}
is uniquely determined.
For $k=0$ the expression \eqref{eq:okropne} is equal to $f$. In the induction step (for $k=0,1,\dots,d-1$) we observe that
$$\sum_{\substack{\mathbf b\in\{L,r\}^d\\(b_1,\dots,b_k)=(a_1,\dots,a_k)}}f_{\mathbf b}=
\sum_{\substack{\mathbf b\in\{L,r\}^d\\(b_1,\dots,b_k)=(a_1,\dots,a_k)\\b_{k+1}=L}}f_{\mathbf b}+\sum_{\substack{\mathbf b\in\{L,r\}^d\\(b_1,\dots,b_k)=(a_1,\dots,a_k)\\b_{k+1}=r}}f_{\mathbf b}$$
is the decomposition of expression \eqref{eq:okropne} to its $\cdot_{(k+1,L)}$ and $\cdot_{(k+1,r)}$ part.
By Remark~\ref{rem:Lr2}, this de\-com\-po\-si\-tion is unique. Consequently, for $k=d$ we obtain the uniqueness of the functions $f_{\mathbf b}$.
\end{proof}


\begin{defin}
For $x,y\in\mathbb R$ and $b\in\{L,r\}$, we define the function $\chi_{x,y}^b\colon\mathbb R\to\mathbb R$ as follows:
$$\chi_{x,y}^L(u)=\mathbf1_{(-\infty,y)}(u)-\mathbf1_{(-\infty,x)}(u)=\begin{cases}1&\text{if }x\leq u<y,\\-1&\text{if }y\leq u<x,\\0&\text{otherwise,}\end{cases}$$
$$\chi_{x,y}^r(u)=\mathbf1_{(-\infty,y]}(u)-\mathbf1_{(-\infty,x]}(u)=\begin{cases}1&\text{if }x<u\leq y,\\-1&\text{if }y<u\leq x,\\0&\text{otherwise.}\end{cases}$$
\end{defin}

\begin{Thm}\label{th:boxmon1}
Let $f\colon\mathbf I\to\mathbb R$ be a~function. Let $\mathbf b=(b_1,b_2,\dots,b_d)\in\{L,r\}^d$ and $\boldsymbol\alpha=(\alpha_1,\alpha_2,\dots,\alpha_d)\in\mathbf I$ be fixed.
Then the following conditions are equivalent:

\begin{itemize}
\item[$\bullet$] The function $f$ is a~$(1,1,\dots,1)$-regular and box-monotone such that for $i=1,2,\dots,d$ and $\mathbf y\in\mathbf I_{\{i\}'}$, we have:
\begin{itemize}
\item[(a)] if $b_i=L$, then $f_{\{i\}}^{\mathbf y}$ is a~left-continuous jump function,
\item[(b)] if $b_i=r$, then $f_{\{i\}}^{\mathbf y}$ is a~right-continuous function,
\item[(c)]  $f_{\{i\}}^{\mathbf y}(\alpha_i)=0$.
\end{itemize}
\item[$\bullet$] There exists a~Borel measure $\mu$ on $\mathbf I$ such that
\begin{itemize}
\item[(i)] $$f(x_1,\dots,x_d)=\idotsint\limits_{\mathbf I}\prod_{j=1}^d\chi_{\alpha_j,x_j}^{b_j}(u_j)\ d\mu(u_1,\dots,u_d),$$
\item[(ii)] $\mu(K)<\infty$ for every compact set $K\subset\mathbf I$,
\item[(iii)] for each $i=1,2,\dots,d$, if $b_i=L$, then the marginal measure of $\mu$, corresponding to the $i$th coordinate is a discrete measure.
\end{itemize}
\end{itemize}
\noindent Moreover, the measure $\mu$ that appears in representation (i) is unique.
\end{Thm}

\begin{proof}
First we prove the implication ($\Leftarrow$).

By (i) and (ii), we obtain that $f(x_1,\dots,x_d)$ is finite for every $(x_1,\dots,x_d)\in\mathbf I$. Moreover, $f(x_1,\dots,x_d)=0$, whenever $x_j=\alpha_j$ for some $j$. Therefore condition (c) is satisfied.

For $y,z\in\mathbb R$, $y<z$, and $b\in\{L,r\}$ we define
\begin{equation}\label{eq:defJ}
J^b_{y,z}=
\begin{cases}
[y,z)&\text{if }b=L,\\
(y,z]&\text{if }b=r.
\end{cases}
\end{equation}

For $j=1,2,\dots,d$ we take arbitrary points $y_j,z_j\in I_j$, $y_j<z_j$.
Then, we have
\begin{equation*}
\begin{split}
\mu\left(\prod_{j=1}^dJ^{b_j}_{y_j,z_j}\right)
=&\idotsint\limits_{\mathbf I}\prod_{j=1}^d\chi_{y_j,z_j}^{b_j}(u_j)\ d\mu(u_1,\dots,u_d)\\
=&\idotsint\limits_{\mathbf I}\prod_{j=1}^d(\chi_{\alpha_j,z_j}^{b_j}(u_j)-\chi_{\alpha_j,y_j}^{b_j}(u_j))\ d\mu(u_1,\dots,u_d)\\
=&\sum_{B\subset\{1,\dots,d\}}(-1)^{|B|}\idotsint\limits_{\mathbf I}\prod_{j=1}^d\chi_{\alpha_j,\hat y_{B,j}}^{b_j}(u_j)\ d\mu(u_1,\dots,u_d)\\
=&\sum_{B\subset\{1,\dots,d\}}(-1)^{|B|}f(\hat y_{B,1},\dots,\hat y_{B,d}),
\end{split}
\end{equation*}
where $\hat y_{B,j}=y_j$ for $j\in B$, and  $\hat y_{B,j}=z_j$ for $j\notin B$. On the other hand, by \eqref{eq:prop5} for $n_1=\dots=n_d=1$, we obtain
$$\left[\ \LARGE{\substack{y_1,\ z_1\\\\y_2,\ z_2\\\\\dots\\\\y_d,\ z_d}};f\right]
=\sum_{B\subset\{1,\dots,d\}}\frac{f(\hat y_{B,1},\dots,\hat y_{B,d})}{\prod_{i=j}^d(\hat y_{B,j}-\hat y_{B',j})}
=\frac1{\prod_{j=1}^d(z_j-y_j)}\cdot\sum_{B\subset\{1,\dots,d\}}(-1)^{|B|}f(\hat y_{B,1},\dots,\hat y_{B,d}).$$
Consequently,
\begin{equation}\label{eq:defmu}
\mu\left(\prod_{j=1}^dJ^{b_j}_{y_j,z_j}\right)
=\sum_{B\subset\{1,\dots,d\}}(-1)^{|B|}f(\hat y_{B,1},\dots,\hat y_{B,d})
=\prod_{j=1}^d(z_j-y_j)\cdot\left[\ \LARGE{\substack{y_1,\ z_1\\\\y_2,\ z_2\\\\\dots\\\\y_d,\ z_d}};f\right].
\end{equation}
It follows that the divided difference that appears in \eqref{eq:defmu} is non-negative. Since $y_j<z_j$ ($j=1,2,\dots,d$) were chosen arbitrarily, the function $f$ is box-monotone.

By the box-monotonicity of $f$, condition (c) and Theorem~\ref{thm:prop8}, we obtain that $f$ is $(1,1,\dots,1)$-regular.

Now, we will show that condition (b) is satisfied.
Let $i=1,2,\dots,d$ be fixed. Assume that $b_i=r$. The functions $g_{x_1,\dots,x_d}\colon\mathbf I\to\mathbb R$ given by formula
$g_{x_1,\dots,x_d}(u_1,\dots,u_d)=\prod_{j=1}^d\chi_{\alpha_j,x_j}^{b_j}(u_j)$ are obviously bounded.
Observe, that if $x_i\downarrow\hat x_i$, then 
$\chi_{\alpha_i,x_i}^r$ is pointwise convergent to $\chi_{\alpha_i,\hat x_i}^r$. Thus, $g_{x_1,\dots,x_i,\dots,x_d}$ converges to $g_{x_1,\dots,\hat x_i,\dots,x_d}$ pointwise.
Consequently, by Lebesgue's dominated convergence theorem, we get $$\lim_{x_i\downarrow \hat x_i}f(x_1,\dots,x_i,\dots,x_d)=f(x_1,\dots,\hat x_i,\dots,x_d),$$
therefore condition (b) holds.

Similarly, one can prove that if $b_i=L$, then $f_{\{i\}}^{\mathbf y}$ is a~left-continuous function for every $\mathbf y\in\mathbf I_{\{i\}'}$.
We will show that $f_{\{i\}}^{\mathbf y}$ is a jump function.
We consider only the case when $y_j\geq\alpha_j$ for all $j\neq i$ (all the other cases are similar). Let $x<y$, $x,y\in I_i$. By condition (iii), we have
\begin{equation*}
\begin{split}
f(y-)-f(x+)=&\mu\left(\prod_{j=1}^{i-1}J^{b_j}_{\alpha_j,x_j}\times(x,y)\times\prod_{j=i+1}^dJ^{b_j}_{\alpha_j,x_j}\right)
=\sum_{x<t<y}\mu\left(\prod_{j=1}^{i-1}J^{b_j}_{\alpha_j,x_j}\times\{t\}\times\prod_{j=i+1}^dJ^{b_j}_{\alpha_j,x_j}\right)\\
=&\sum_{x<t<y}(f(t+)-f(t-)).
\end{split}
\end{equation*}
We obtained that equation \eqref{eq:jump} holds for each $x<y$, which implies that $f_{\{i\}}^{\mathbf y}$ is a jump function.
This completes the proof of condition (a).


Now, we are going to the proof of the implication ($\Rightarrow$).

We need to define the measure $\mu$. First, by formula \eqref{eq:defmu}, we define $\mu$ for the rectangles of the form $\prod_{j=1}^dJ^{b_j}_{y_j,z_j}$, where $y_j,z_j\in I_j$, $y_j<z_j$, and $J^b_{y,z}$ is given by \eqref{eq:defJ}. Using the standard methods, we may uniquely extend $\mu$ to a Borel measure on $\mathbf I$. In this step, the assumptions on one-sided continuity of $f$ ((a) and (b)) play an important role. We omit the details of the proof of the extension.
Note, that every measure $\mu$ satisfying condition (i) needs to fulfill equality \eqref{eq:defmu} for all rectangles. Thus, $\mu$ is uniquely determined.

If $K\subset\mathbf I$ is compact, then it is a~subset of some rectangle $\mathbf J=\prod_{j=1}^dJ^{b_j}_{y_j,z_j}$. Therefore, $\mu(K)\leq\mu(\mathbf J)<\infty$, which proves (ii).

We will show (i). Let $(x_1,x_2,\dots,x_d)\in\mathbf I$.
For each $j=1,2,\dots,d$ we put $(y_j,z_j)=(\min(x_j,\alpha_j),\max(x_j,\alpha_j))$. Note that if $y_j=\alpha_j$, then $\chi_{\alpha_j,x_j}^{b_j}=\chi_{y_j,z_j}^{b_j}$. Otherwise, we have $\chi_{\alpha_j,x_j}^{b_j}=-\chi_{y_j,z_j}^{b_j}$.
Consequently, letting $k=|\{j\colon x_j<\alpha_j\}|$, we obtain
\begin{equation*}
\begin{split}
\idotsint\limits_{\mathbf I}&\prod_{j=1}^d\chi_{\alpha_j,x_j}^{b_j}(u_j)\ d\mu(u_1,\dots,u_d)
=(-1)^k\idotsint\limits_{\mathbf I}\prod_{j=1}^d\chi_{y_j,z_j}^{b_j}(u_j)\ d\mu(u_1,\dots,u_d)\\
=&(-1)^k\mu\left(\prod_{j=1}^dJ^{b_j}_{y_j,z_j}\right)
=(-1)^k\sum_{B\subset\{1,\dots,d\}}(-1)^{|B|}f(\hat y_{B,1},\dots,\hat y_{B,d}).
\end{split}
\end{equation*}
Using (c), we obtain that  the only non-zero summand in the last sum is the one for $B=\{j\colon x_j<\alpha_j\}$. This summand equals $(-1)^{|B|}f(\hat y_{B,1},\dots,\hat y_{B,d})=(-1)^kf(x_1,x_2,\dots,x_d)$.
Therefore, we obtain (i).

It remains to show (iii). We fix $i=1,2,\dots,d$ such that $b_i=L$. By the definition of $\mu$ and formula \eqref{eq:defmu}, condition (a) implies that the function $z_i\mapsto\mu\left(\prod_{j=1}^dJ^{b_j}_{y_j,z_j}\right)$ is a jump function. Thus, we have that the $i$th marginal of $\mu$ restricted to any rectangle of the form $\prod_{j=1}^dJ^{b_j}_{y_j,z_j}$ is a discrete measure. Consequently, condition (iii) is satisfied.
\end{proof}

\begin{Thm}\label{th:boxmonint}
Let $f\colon\mathbf I\to\mathbb R$ be a function. Let $\boldsymbol\alpha=(\alpha_1,\alpha_2,\dots,\alpha_d)\in\mathbf I$ be fixed.
Then the following conditions are equivalent:
\begin{itemize}
\item[$\bullet$] $f$ is box-monotone,
\item[$\bullet$] there exist a~pseudo-polynomial $W\colon\mathbf I\to\mathbb R$ of the degree $(0,0,\dots,0)$ and Borel measures $\mu_{\mathbf b}$ on $\mathbf I$ $(\mathbf b\in\{L,r\}^d)$ such that
\begin{itemize}
\item[(i)]
\begin{equation}\label{eq:boxmon}
f(x_1,\dots,x_d)=W(x_1,\dots,x_d)+\sum_{\mathbf b\in\{L,r\}^d}\idotsint\limits_{\mathbf I}\prod_{j=1}^d\chi_{\alpha_j,x_j}^{b_j}(u_j)\ d\mu_{\mathbf b}(u_1,\dots,u_d),
\end{equation}
\item[(ii)] $\mu_{\mathbf b}(K)<\infty$ for every compact set $K\subset\mathbf I$ and $\mathbf b\in\{L,r\}^d$,
\item[(iii)] for each $\mathbf b\in\{L,r\}^d$ and $i=1,2,\dots,d$, if $b_i=L$, then the marginal measure of $\mu_{\mathbf b}$, corresponding to the $i$th coordinate is a discrete measure.
\end{itemize}
\end{itemize}

\noindent Moreover, for a box-monotone function $f$, the pseudopolynomial $W$ and the measures $\mu_{\mathbf b}$ in the representation \eqref{eq:boxmon} are uniquely determined.
\end{Thm}
\begin{proof}
If the function $f$ is of the form \eqref{eq:boxmon}, then, by Theorems~\ref{thm:pseudo} and \ref{th:boxmon1}, it is box-monotone as a sum of box-monotone functions. The proof of the implication ($\Leftarrow$) is completed.

The implication ($\Rightarrow$) is an immediate consequence of Lemma~\ref{lem:pseudo2} and Theo\-rems~\ref{thm:prop8}, \ref{th:rozklad} and \ref{th:boxmon1}.

Now, we will prove the uniqueness of the representation. Assume that \eqref{eq:boxmon} holds both for $W$, $\mu_{\mathbf b}$ and $\hat W$, $\hat \mu_{\mathbf b}$ ($\mathbf b\in\{L,r\}^d$). Then we have
$$(\hat W-W)(x_1,\dots,x_d)=\sum_{\mathbf b\in\{L,r\}^d}\idotsint\limits_{\mathbf I}\prod_{j=1}^d\chi_{\alpha_j,x_j}^{b_j}(u_j)\ d(\mu_{\mathbf b}-\hat\mu_{\mathbf b})(u_1,\dots,u_d).$$
We obtain that $\hat W-W$ is
a~pseudo-polynomial of the degree $(0,0,\dots,0)$ such that $(\hat W-W)(x_1,\dots,x_d)=0$, whenever $x_i=\alpha_i$ for some $i=1,2,\dots,d$. By Proposition~\ref{prop:prop6}, it follows that $W=\hat W$. Consequently, by Theorem~\ref{th:rozklad}, all the summands on the right side are zero functions, and by Theorem~\ref{th:boxmon1}, $\mu_{\mathbf b}=\hat\mu_{\mathbf b}$ for all $\mathbf b\in\{L,r\}^d$. This completes the proof of the theorem.
\end{proof}

\begin{Rem}
Note, that each $\mathbf b\in\{L,r\}^d$ the summand corresponding to $\mathbf b$ in the representation \eqref{eq:boxmon} can be regarded as a~generalized cumulative distribution function of the measure $\mu_{\mathbf b}$.

Let us consider the case, when all the measures $\mu_{\mathbf b}$ are finite, which is equivalent to the existence of $M\in\mathbb R$ such that for arbitrary points $y_j,z_j\in I_j$, $j=1,2,\dots,d$ we have
\begin{equation}\label{eq:finite}
\prod_{j=1}^d(z_j-y_j)\cdot\left[\ \LARGE{\substack{y_1,\ z_1\\\\y_2,\ z_2\\\\\dots\\\\y_d,\ z_d}};f\right]
=\sum_{B\subset\{1,\dots,d\}}(-1)^{|B|}f(\hat y_{B,1},\dots,\hat y_{B,d})<M,
\end{equation}
where $\hat y_{B,j}=y_j$ for $j\in B$, and $\hat y_{B,j}=z_j$ for $j\notin B$. In that case formula~\eqref{eq:boxmon} can be written using probabilistic tools.

Let $\mathbf b\in\{L,r\}^d$ and let $c_{\mathbf b}=\mu_{\mathbf b}(\mathbf I)<\infty$. There exists a~random vector $\mathbf X_{\mathbf b}=(X_{\mathbf b,1},\dots,X_{\mathbf b,d})$ such that $\mu_{\mathbf b}/c_{\mathbf b}$ is the probability distribution of the vector $\mathbf X_{\mathbf b}$, whenever $c_{\mathbf b}>0$. We have
\begin{equation}\label{eq:random}
\begin{split}
\idotsint\limits_{\mathbf I}&\prod_{j=1}^d\chi_{\alpha_j,x_j}^{b_j}(u_j)\ d\mu_{\mathbf b}(u_1,\dots,u_d)
=c_{\mathbf b}\idotsint\limits_{\mathbf I}\prod_{j=1}^d(\mathbf1_{(-\infty,x_j\rangle}(u_j)-\mathbf1_{(-\infty,\alpha_j\rangle}(u_j))\ dP_{\mathbf X_{\mathbf b}}(u_1,\dots,u_d)\\
=&c_{\mathbf b}
\sum_{B\subset\{1,\dots,d\}}(-1)^{|B|}
P(X_{\mathbf b,1}<_{b_1}\hat x_{B,1},\dots,X_{\mathbf b,d}<_{b_d}\hat x_{B,d}),
\end{split}
\end{equation}
where the bracket $\rangle$ in the $j$th factor is either $)$ if $b_j=L$ or $]$ if $b_j=r$, the symbol $<_L$ denotes $<$, the symbol $<_r$ denotes $\leq$; $\hat x_{B,j}=\alpha_j$ for $j\in B$, and $\hat x_{B,j}=x_j$ for $j\notin B$.

If $B\neq\emptyset$, then there exists $i\in B$, and the corresponding summand in \eqref{eq:random} does not depend on $x_i$. Consequently, this summand is a~box-$(1,\dots,1)$-affine function.
It follows, that the expression \eqref{eq:random} is a sum of the~pseudo-polynomial of degree $(0,\dots,0)$ and the summand corresponding to $B=\emptyset$, namely ${c_{\mathbf b}\cdot P}(X_{\mathbf b,1}<_{b_1} x_1,\dots,X_{\mathbf b,d}<_{b_d}x_d)$.

Finally we obtain that every box-monotone function $f$ satisfying condition \eqref{eq:finite} (e.g. every bounded box-monotone function $f$) has the following
proba\-bi\-li\-stic representation
\begin{equation}\label{eq:random2}
f(x_1,\dots,x_d)=V(x_1,\dots,x_d)+\sum_{\mathbf b\in\{L,r\}^d}c_{\mathbf b}\cdot P(X_{\mathbf b,1}<_{b_1} x_1,\dots,X_{\mathbf b,d}<_{b_d}x_d),
\end{equation}
where $V$ is a~pseudo-polynomial of degree $(0,\dots,0)$, and for every $\mathbf b\in\{L,r\}^d$ $\mathbf X_{\mathbf b}=(X_{\mathbf b,1},\dots,X_{\mathbf b,d})$ is a~random vector and $c_{\mathbf b}\geq0$. On the other hand, it can be shown that every function $f$ of the form \eqref{eq:random2} is box-monotone and it satisfies condition \eqref{eq:finite}.
\end{Rem}

\section{Integral representation of box-$\textbf{n}$-convex functions}

\begin{Thm}\label{th:bismnrepres}
Let $I_1, I_2,\dots,I_d$ be open intervals (bounded or unbounded) and $f\colon\mathbf I\to\mathbb R$ be a function. Let $\boldsymbol\alpha=(\alpha_1,\alpha_2,\dots,\alpha_d)\in\mathbf I$ and $\mathbf n=(n_1,n_2,\dots,n_d)\in\mathbb N^d$ be such that $n_i\geq1$ for $i=1,2,\dots,d$.

Then $f$ is box-$\mathbf n$-convex, if and only if it is of the form 
\begin{equation}\label{eq:represtresbis}
f(x_1,\dots,x_d)=W(x_1,\dots,x_d)
+\sum_{\mathbf b\in\{L,r\}^d}\idotsint\limits_{\mathbf I}\prod_{j=1}^d\frac{(x_j-u_j)^{n_j-1}}{(n_j-1)!}\chi_{\alpha_j,x_j}^{b_j}(u_j)\ d\mu_{\mathbf b}(u_1,\dots,u_d),
\end{equation}
where $W\colon\mathbf I\to\mathbb R$ is a pseudo-polynomial of degree $(n_1-1,\dots,n_d-1)$ and measures $\mu_{\mathbf b}$ $(\mathbf b\in\{L,r\}^d)$ satisfy conditions (ii) and (iii) of Theorem~\ref{th:boxmonint} $($here we use the standard convention that $0^0=1)$.
\end{Thm}

\begin{proof}
We show the implication ($\Rightarrow$). Taking into account Proposition~\ref{prop:regularize}, it is enough to prove the implication in the case, when the function $f$ is $\mathbf n$-regular. We use an induction on $n_1+\dots+n_d$. If $n_1+\dots+n_d=d$, then \eqref{eq:represtresbis} is a consequence of Theorem~\ref{th:boxmonint}.
If $n_1+\dots+n_d>d$, then $n_i\geq2$ for some $i$. By Lemma~\ref{lem:lem15v2}, we have $f(x_1,\dots,x_d)=f(x_1,\dots,\alpha_i,\dots,x_d)+\int_{\alpha_i}^{x_i} \psi(x_1,\dots,t,\dots,x_d)dt$, where $t$ and $\alpha_i$ stand at the $i$th position, and $\psi\colon\mathbf I\to\mathbb R$ is a~$(n_1,\dots,n_i-1,\dots,n_d)$-regular and box-$(n_1,\dots,n_i-1,\dots,n_d)$-convex function.
By the induction hypothesis, we have
\begin{multline*}
\psi(x_1,\dots,x_d)=V(x_1,\dots,x_d)\\
+\sum_{\mathbf b\in\{L,r\}^d}\idotsint\limits_{\mathbf I}
\frac{(x_i-u_i)^{n_i-2}}{(n_i-2)!}\chi_{\alpha_i,x_i}^{b_i}(u_i)\prod_{j\neq i}\frac{(x_j-u_j)^{n_j-1}}{(n_j-1)!}\chi_{\alpha_j,x_j}^{b_j}(u_j)\ d\mu_{\mathbf b}(u_1,\dots,u_d).
\end{multline*}
Taking $W(x_1,\dots,x_d)=f(x_1,\dots,\alpha_i,\dots,x_d)+\int_{\alpha_i}^{x_i}V(x_1,\dots,t,\dots,x_d)dt$ and observing that
$$\int_{\alpha_i}^{x_i}\frac{(t-u_i)^{n_i-2}}{(n_i-2)!}\chi_{\alpha_i,t}^{b_i}(u_i)dt=\frac{(x_i-u_i)^{n_i-1}}{(n_i-1)!}\chi_{\alpha_i,x_i}^{b_i}(u_i),$$ we obtain \eqref{eq:represtresbis}.

We pass to the proof of the implication ($\Leftarrow$). We assume that the function $f$ is of the form \eqref{eq:represtresbis}. It is enough to prove that each of the summands is box-$\mathbf n$-convex. We will proceed by induction on $n_1+\dots+n_d$. If $n_1+\dots+n_d=d$, then the box-$\mathbf n$-convexity of each summand is a consequence of Theorem~\ref{th:boxmonint}. If $n_1+\dots+n_d>d$, then $n_i\geq2$ for some $i$. The box-$\mathbf n$-convexity of

\begin{multline*}
\idotsint\limits_{\mathbf I}\prod_{j=1}^d\frac{(x_j-u_j)^{n_j-1}}{(n_j-1)!}\chi_{\alpha_j,x_j}^{b_j}(u_j)\ d\mu_{\mathbf b}(u_1,\dots,u_d)\\\\
=\int_{\alpha_i}^{x_i}\idotsint\limits_{\mathbf I}
\frac{(t-u_i)^{n_i-2}}{(n_i-2)!}\chi_{\alpha_i,t}^{b_i}(u_i)\prod_{j\neq i}\frac{(x_j-u_j)^{n_j-1}}{(n_j-1)!}\chi_{\alpha_j,x_j}^{b_j}(u_j)\ d\mu_{\mathbf b}(u_1,\dots,u_d)dt
\end{multline*}
is a consequence of the induction hypothesis and Lemma~\ref{lem:calka}.
The theorem is proved.
\end{proof}

The above integral representation becomes much simpler when all $n_i$'s are greater than $1$.

\begin{Thm}\label{th:mnrepres}
Let $I_1, I_2,\dots,I_d$ be open intervals (bounded or unbounded) and $f\colon\mathbf I\to\mathbb R$ be a function. Let $\boldsymbol\alpha=(\alpha_1,\alpha_2,\dots,\alpha_d)\in\mathbf I$ and $\mathbf n=(n_1,n_2,\dots,n_d)\in\mathbb N^d$ be such that $n_i\geq2$ for $i=1,2,\dots,d$.

Then $f$ is box-$\mathbf n$-convex, if and only if it is of the form 
\begin{equation}\label{eq:repres2}
f(x_1,\dots,x_d)=W(x_1,\dots,x_d)
+\idotsint\limits_{\mathbf I}\prod_{j=1}^d\frac{(x_j-u_j)^{n_j-1}}{(n_j-1)!}\chi_{\alpha_j,x_j}^r(u_j)\ d\mu(u_1,\dots,u_d),
\end{equation}
where $W\colon\mathbf I\to\mathbb R$ is a~pseudo-polynomial of degree $(n_1-1,\dots,n_d-1)$ and the~measure $\mu$ is finite on compact sets.
\end{Thm}

\begin{proof}
The implication ($\Leftarrow$) follows from Theorem~\ref{th:bismnrepres}.

We prove the implication $(\Rightarrow$).
By Theorem~\ref{th:bismnrepres}, the function $f$ is of the form \eqref{eq:represtresbis}.
Let $\mathbf b\in\{L,r\}^d$.
Let $\mathbf1_{b=L}=1$ if $b=L$, and $\mathbf1_{b=L}=0$ if $b=r$.
Since $n_1,\dots,n_d\geq2$, we have
\begin{equation*}
\frac{(x_j-u_j)^{n_j-1}}{(n_j-1)!}\chi_{\alpha_j,x_j}^{b_j}(u_j)
=\frac{(x_j-u_j)^{n_j-1}}{(n_j-1)!}\chi_{\alpha_j,x_j}^r(u_j)+\frac{(x_j-u_j)^{n_j-1}}{(n_j-1)!}\mathbf1_{\{\alpha_j\}}(u_j)\cdot\mathbf1_{b_j=L}.
\end{equation*}
Using the above identity we obtain
\begin{equation}\label{eq:boxrep}
\begin{split}
\idotsint\limits_{\mathbf I}&\prod_{j=1}^d\frac{(x_j-u_j)^{n_j-1}}{(n_j-1)!}\chi_{\alpha_j,x_j}^{b_j}(u_j)\ d\mu_{\mathbf b}(u_1,\dots,u_d)\\
=&\sum_{B\subset\{1,\dots,d\}}\idotsint\limits_{\mathbf I}\prod_{j=1}^d\frac{(x_j-u_j)^{n_j-1}}{(n_j-1)!}\hat\chi_{B,j}(u_j)\ d\mu_{\mathbf b}(u_1,\dots,u_d),
\end{split}
\end{equation}
where $\hat\chi_{B,j}(u_j)=\mathbf1_{\{\alpha_j\}}(u_j)\cdot\mathbf1_{b_j=L}$ for $j\in B$, and $\hat\chi_{B,j}(u_j)=\chi_{\alpha_j,x_j}^r(u_j)$ for $j\notin B$.

If $B\neq\emptyset$, then there exists $i\in B$, and the corresponding summand in \eqref{eq:boxrep} can be written as
\begin{equation*}
\begin{split}
&\idotsint\limits_{\mathbf I}\prod_{j=1}^d\frac{(x_j-u_j)^{n_j-1}}{(n_j-1)!}\hat\chi_{B,j}(u_j)\ d\mu_{\mathbf b}(u_1,\dots,u_d)\\
=\frac{(x_i-\alpha_i)^{n_i-1}}{(n_i-1)!}
&\cdot\idotsint\limits_{\mathbf I}\mathbf1_{\{\alpha_i\}}(u_i)\cdot\mathbf1_{b_i=L}\prod_{j\neq i}\frac{(x_j-u_j)^{n_j-1}}{(n_j-1)!}\hat\chi_{B,j}(u_j)\ d\mu_{\mathbf b}(u_1,\dots,u_d).
\end{split}
\end{equation*}
Since the last integral does not depend on $x_i$, we obtain that the above formula represents a~pseudo-polynomial of degree $(n_1-1,\dots,n_d-1)$.

It follows that the expression \eqref{eq:random} is a sum of a~pseudo-polynomial of degree $(n_1-1,\dots,n_d-1)$ and the summand corresponding to $B=\emptyset$, namely $\idotsint\limits_{\mathbf I}\prod_{j=1}^d\frac{(x_j-u_j)^{n_j-1}}{(n_j-1)!}\chi_{\alpha_j,x_j}^r(u_j)\ d\mu_{\mathbf b}(u_1,\dots,u_d)$.
Consequently, the function $f$ is of the form \eqref{eq:repres2}, where $\mu=\sum_{\mathbf b\in\{L,r\}^d}\mu_{\mathbf b}$. The theorem is proved.
\end{proof}


Using the same method as in the proof of Theorem~\ref{th:mnrepres}, we can prove the following corollary.

\begin{Cor}\label{cor:mnrepres}
Let $I_1, I_2,\dots,I_d$ be open intervals (bounded or unbounded) and $f\colon\mathbf I\to\mathbb R$ be a function. Let $\boldsymbol\alpha=(\alpha_1,\alpha_2,\dots,\alpha_d)\in\mathbf I$ and $\mathbf n=(n_1,n_2,\dots,n_d)\in\mathbb N^d$ be such that $n_i\geq1$ for $i=1,2,\dots,d$. We put $\mathbf B=\{(b_1,\dots,b_d)\in\{L,r\}^d:b_i=r$ if $n_i\geq2, i=1,\dots,d\}$.

Then $f$ is box-$\mathbf n$-convex, if and only if it is of the form 
\begin{equation}\label{eq:represbis100}
f(x_1,\dots,x_d)=W(x_1,\dots,x_d)
+\sum_{\mathbf b\in\mathbf B}\idotsint\limits_{\mathbf I}\prod_{j=1}^d\frac{(x_j-u_j)^{n_j-1}}{(n_j-1)!}\chi_{\alpha_j,x_j}^{b_j}(u_j)\ d\mu_{\mathbf b}(u_1,\dots,u_d),
\end{equation}
where $W\colon\mathbf I\to\mathbb R$ is a pseudo-polynomial of degree $(n_1-1,\dots,n_d-1)$ and measures $\mu_{\mathbf b}$ $(\mathbf b\in\mathbf B)$ satisfy conditions (ii) and (iii) of Theorem~\ref{th:boxmonint} $($here we use the standard convention that $0^0=1)$.
\end{Cor}

\begin{Rem}
The representation \eqref{eq:represbis100} is a counterpart of the representation \eqref{eq:repres2} in the case when some (not necessarily all) of $n_i$'s are greater than $1$. In particular, Theorem~\ref{th:mnrepres} is a~special case of Corollary~\ref{cor:mnrepres}.
It can be shown that for given $f$ and $\boldsymbol\alpha$ the pseudo-polynomial $W$ and the measures
$\mu_{\mathbf b}$ ($\mathbf b\in\mathbf B$) are uniquely determined.
\end{Rem}

\begin{Rem}
In the above results, we focused on the box-$\mathbf n$-convex functions, in the case when $n_i\geq1$ for all $i=1,2,\dots,d$.
We can also consider the case when $n_i=0$ for some $i$.

Let $f\colon\mathbf I\to\mathbb R$ and $\mathbf n=(n_1,n_2,\dots,n_d)\in\mathbb N^d$ be such that $n_i\geq0$ for $i=1,2,\dots,d$.
Let $A=\{i=1,2,\dots,d:n_i\geq1\}$.
Then the function $f$ is box-$\mathbf n$-convex if and only if for every $\mathbf z\in\mathbf I_{A'}$ the function $f_A^{\mathbf z}$ is box-$\mathbf n_A$-convex.
Since all the coordinates of $\mathbf n_A$ are positive, the box-$\mathbf n_A$-convexity is characterized in Corollary~\ref{cor:mnrepres} and the preceding results.
\end{Rem}

Let $x_+=\max\{0,x  \}$ and $x_-=\max\{0,-x  \}$, then $x=x_+ - x_-$ $(x\in \R)$. It is not difficult to prove the following version of Theorem \ref{th:mnrepres}, which is a~$d$-dimensional counterpart of the well-known integral spline representation of $n$-convex functions in one dimension.

\begin{Thm}\label{th:repres_int_v3}
Let $I_1, I_2,\dots,I_d$ be open intervals (bounded or unbounded). Let $\mathbf n=(n_1,n_2,\dots,n_d)\in\mathbb N^d$ be such that $n_i\geq2$ for $i=1,2,\dots,d$.

Let $f\colon\mathbf I\to\mathbb R$ be a~function. Then $f$ is box-$\mathbf n$-convex, if and only if for every $\boldsymbol\alpha=(\alpha_1,\alpha_2,\dots,\alpha_d)\in\mathbf I$, there exist 
a~pseudo-polynomial $W_{\boldsymbol\alpha}$ of degree $(n_1-1,\dots,n_d-1)$ and 
a~Borel measure $\mu$ on $\mathbf I$, 
 which is finite on convex sets, 
such that $f$ is of the form
\begin{multline}\label{eq:repres_v10}
f(x_1,\dots,x_d)=W_{\boldsymbol\alpha}(x_1,\dots,x_d)\\
+\sum_{A\subset\{1,2,\ldots,d
  \}}\idotsint\limits_{\mathbf I_{A,\boldsymbol\alpha}}
  \prod_{j\in A'}\frac{(x_j-u_j)_+^{n_j-1}}{(n_j-1)!}\prod_{j\in A}(-1)^{n_j}\frac{(x_j-u_j)_-^{n_j-1}}{(n_j-1)!}\ d\mu(u_1,\dots,u_d),
\end{multline}
where $\mathbf I_{A,\boldsymbol\alpha}=\{(u_1,\ldots,u_d )\in \textbf{I}\colon u_j\leq\alpha _j \:if\: j\in A\: and\: u_j>\alpha _j\: if\: j \in A'  \}$.
\end{Thm}

\begin{Thm}
Let $I_1, I_2,\dots,I_d$ be open intervals (bounded or unbounded). Let $\mathbf n=(n_1,n_2,\dots,n_d)\in\mathbb N^d$ be such that $n_i\geq2$ for $i=1,2,\dots,d$.

Let $f\colon\mathbf I\to\mathbb R$ be a~function. Then $f$ is box-$\mathbf n$-convex, if and only if for all $\boldsymbol\alpha=(\alpha_1,\alpha_2,\dots,\alpha_d)\in\mathbf I$, and $\mathbf J^{\boldsymbol\alpha}=\mathbf I\cap\prod_{i=1}^d(\alpha_i,\infty)$, we have 
\begin{equation*}
f(x_1,\dots,x_d)=W_{\boldsymbol\alpha}(x_1,\dots,x_d)
+\idotsint\limits_{\mathbf J^{\boldsymbol\alpha}}\prod_{j=1}^d\frac{(x_j-u_j)_+^{n_j-1}}{(n_j-1)!}\ d\mu_{\boldsymbol\alpha}(u_1,\dots,u_d),
\end{equation*}
for $(x_1,\dots,x_d)\in\mathbf J^{\boldsymbol\alpha}$, where $W_{\boldsymbol\alpha}\colon\mathbf J^{\boldsymbol\alpha}\to\mathbb R$ is a~pseudo-polynomial of degree $(n_1-1,\dots,n_d-1)$, and $\mu_{\boldsymbol\alpha}$ is a~Borel measure on $\mathbf I$, which is finite on convex sets.
\end{Thm}

\begin{proof}
The implication ($\Rightarrow$) is an immediate consequence of Theorem~\ref{th:mnrepres} (we take $W_{\boldsymbol\alpha}=W|_{\mathbf J^{\boldsymbol\alpha}}$ and $\mu_{\boldsymbol\alpha}=\mu$).

We show the implication ($\Leftarrow$).
By Theorem~\ref{th:mnrepres},
the function
\begin{equation*}
g(x_1,x_2,\dots,x_d)=\idotsint\limits_{\mathbf I}\prod_{j=1}^d\frac{(x_j-u_j)^{n_j-1}}{(n_j-1)!}\chi_{\alpha_j,x_j}^r(u_j)\ d\mu_{\boldsymbol\alpha}(u_1,\dots,u_d)
\end{equation*}
is box-$\mathbf n$-convex on $\mathbf I$. As a consequence, it is also box-$\mathbf n$-convex on $\mathbf J^{\boldsymbol\alpha}$. Conse\-quen\-tly, the function $f|_{\mathbf J^{\boldsymbol\alpha}}=W_{\boldsymbol\alpha}+g$ is box-$\mathbf n$-convex.

Since for every $\boldsymbol\alpha\in\mathbf I$, $f$ is box-$\mathbf n$-convex on $\mathbf J^{\boldsymbol\alpha}$, we obtain that $f$ is box-$\mathbf n$-convex on $\mathbf I$.
\end{proof}


\section {Box-$\mathbf n$-convex orders.}
By the analogy to the $n$-convex orders and  the box-$(m,n)$-convex orders \cite{DLS 1998,KomRaj2023,Shaked2007}, we will define the box-$\mathbf n$-convex orders.

First, let us recall the definition of the $n$-convex order.
\begin{defin} 
Let $X$ and $Y$ be two random variables that take on values in the interval $I\subset \R$. Then $X$ is said to be smaller than $Y$ in the $n$-convex (respectively, $n$-concave) order, denoted by $X\leq _{n\mhyphen  cx}Y $ ($X\leq _{n\mhyphen  cv}Y $), if
$$
\E f(X) \leq \E f(Y) 
$$
for all $n$-convex ($n$-concave) functions $f:I \to\mathbb R$, for which the expectations exist.
\end{defin}
Many properties of the $n$-convex order can be found in \cite{DLS 1998}. 
\begin{Thm}\label{th:theorem100}
Let $X$ and $Y$ be two $I$-valued  random variables such that 
$\E |X|^n<\infty $ and $\E |Y|^n<\infty $. Then $X\leq _{n\mhyphen  cx}Y $ if, and only if, 
\begin{itemize}
\item[{\rm\textbf{a)}}]
$\E X^k= \E Y^k$, $k=1, \ldots, n$,
\item[{\rm\textbf{b)}}]
$\E (X-t)^n_+\leq \E (Y-t)^n_+$ for all $t\in I$.
\end {itemize}
\end{Thm}
\begin{Thm}\label{th:theorem101}
Let $X$ and $Y$ be two $I$-valued  random variables such that 
$\E |X|^n<\infty $ and $\E |Y|^n<\infty $. Then $X\leq _{n\mhyphen  cx}Y $ if, and only if, 
\begin{itemize}
\item[{\rm\textbf{a)}}]
$\E X^k= \E Y^k$, $k=1, \ldots, n$,
\item[{\rm\textbf{b')}}]
$\E (-1)^{n+1}(X-t)^n_-\leq \E (-1)^{n+1}(Y-t)^n_-$ for all $t\in I$.
\end {itemize}
\end{Thm}
It is not difficult to prove the following lemma.
\begin{Lem}\label{lem:lem102}
Let $x,u\in\R$, $n\in\N$, then $(x-u)^n_+ = (x-u)^n + (-1)^{n+1}(x-u)^n_-$.
\end{Lem}
Note, that Theorem \ref{th:theorem101} easily follows from Theorem \ref{th:theorem100} and Lemma \ref{lem:lem102}.

By analogy to the $n$-convex order for random variables, one can define the $n$-convex order for the signed measures $\gamma_1$, $\gamma_2$ on $I$, which have a finite variation: $\gamma_1 \leq _{n\mhyphen  cx} \gamma_2$ if $\int _I f(x)d\gamma_1(x)\leq\int _I f(x)d\gamma_2(x)$ for all for all  $n$-convex functions $f\colon I \to\mathbb R$, provided the integrals exist. Similarly one can define the $n$-concave order $(\leq _{n\mhyphen  cv})$ for signed measures. 

By Theorem \ref{th:theorem100} (cf \cite{Rajba2017}), it follows the following characterization of $n$-convex orders for the signed measures.
\begin{Thm}\label{th:theorem103}
Let $\gamma$ be a signed measure on $I$, which has a finite variation and such that $\int _I|x|^n d |\gamma|(x) <\infty$.
Then $\gamma\geq _{n\mhyphen  cx}0 $ if, and only if, 
\begin{itemize}
\item[{\rm\textbf{a)}}]
$\int_I x^k d \gamma(x)=0$, $k=1, \ldots, n$, 
\vspace{1ex}
\item[{\rm\textbf{b)}}]
$\int_I (x-u)^n_+d\gamma(x)\geq 0$ for all $u\in I$.
\end {itemize}
\end{Thm}
\begin{Rem}\label{rem:rem104}
By Theorem \ref{th:theorem101}, it follows, that condition (b) in Theorem \ref{th:theorem103}, can be replaced by 
\begin{itemize}
\item[{\rm\textbf{b')}}]
$\int_I (-1)^{n+1}(x-u)^n_-d\gamma(x)\geq 0$ for all $u\in I$.
\end {itemize}
\end {Rem}
\begin{Rem}\label{rem:rem105}
Similarly, one can give characterizations of $n$-concave orders if conditions (b) and (b') (in Theorems \ref{th:theorem100}, \ref{th:theorem101}, \ref{th:theorem103} and Remark \ref{rem:rem104}) hold with the sign reversed.
\end {Rem}

By analogy to the $n$-convex orders and  the box-$(m,n)$-convex orders, we define the box-$\mathbf n$-convex orders.

\begin{defin}
Let $(X_1, \ldots, X_d)$ and $(Y_1, \ldots, Y_d)$ be two $\mathbf{I}$-valued random vectors. Then $(X_1, \ldots, X_d)$ is said to be smaller then  $(Y_1, \ldots, Y_d)$ in the box-$\mathbf{n}$-convex order, denoted by $(X_1, \ldots, X_d)\prec_{box\mhyphen \mathbf n\mhyphen cx}(Y_1, \ldots, Y_d)$, if 
\begin{equation}\label{eq:eq100}
 \E f(X_1, \ldots, X_d)\leq \E f(Y_1, \ldots, Y_d)  
\end {equation}
for all continuous box-$\mathbf n$-convex functions $f\colon \mathbf{I} \to\mathbb R$, provided the expectations exist. 
\end{defin}
\begin{Thm}\label{th:theorem106}
Let $(X_1, \ldots, X_d)$ and $(Y_1, \ldots, Y_d)$ be two $\mathbf{I}$-valued random vec\-tors such that $\E |X_1|^{n_1-1}\ldots |X_d|^{n_d-1}<\infty$ and $\E |Y_1|^{n_1-1}\ldots |Y_d|^{n_d-1}<\infty$. Then
\begin{equation}\label{eq:eq100a}
(X_1, \ldots, X_d)\prec_{box\mhyphen \mathbf n\mhyphen cx}(Y_1, \ldots, Y_d)
\end{equation}
if, and only if, the following conditions are satisfied
\begin{itemize}
\item[{\rm\textbf{a)}}]
\begin{equation}\label{eq:eq101}
\E W(X_1, \ldots, X_d)=\E W(Y_1, \ldots, Y_d)
\end{equation}
for all continuous pseudo-polynomials $W$ of degree $(n_1-1, \ldots n_d-1)$, provided the expectations exist,
\item[{\rm\textbf{b)}}]
\begin{equation}\label{eq:eq102}
\E \prod_{j\in A'}\frac{(X_j-u_j)_+^{n_j-1}}{(n_j-1)!}\prod_{j\in A}(-1)^{n_j}\frac{(X_j-u_j)_-^{n_j-1}}{(n_j-1)!}
\leq \E \prod_{j\in A'}\frac{(Y_j-u_j)_+^{n_j-1}}{(n_j-1)!}\prod_{j\in A}(-1)^{n_j}\frac{(Y_j-u_j)_-^{n_j-1}}{(n_j-1)!}
\end{equation}
for all $A\subset\{1,2,\ldots,d\}$ and $(u_1, \ldots, u_d)\in \mathbf I$.

\end {itemize}
\end{Thm}
\begin{proof}
We show the implication ($\Rightarrow$). Let $W$ be a continuous pseudo-polynomial of degree $(n_1-1, \ldots n_d-1)$. By Lemma \ref{lem:pseudo3}, $W$ is box-$\mathbf n$-affine. Then taking in \eqref{eq:eq100} $f=W$ and next $f=-W$, we obtain \eqref{eq:eq101}.

Let $A\subset\{1,2,\ldots,d\}$ and $\mathbf{u}=(u_1, \ldots, u_d)\in \mathbf{I}$ be fixed. Then there exists $(\alpha_1, \ldots, \alpha_d)$ $\in \mathbf{I}$ such that $\alpha_j<u_j$ if $j\in A'$ and $\alpha_j>u_j$ if $j\in A$, consequently $\mathbf{u}\in\mathbf I_{A,(\alpha_1, \ldots, \alpha_d)}$. Let $f_{A,\mathbf{u}}\colon\mathbf I\to\mathbb R$ be the function given by the formula
\begin{equation*}
f_{A,\mathbf{u}}(x_1, \ldots, x_d)=
\prod_{j\in A'}\frac{(x_j-u_j)_+^{n_j-1}}{(n_j-1)!}\prod_{j\in A}(-1)^{n_j}\frac{(x_j-u_j)_-^{n_j-1}}{(n_j-1)!}
\end{equation*}
Then, by Theorem \ref{th:repres_int_v3}, the function $f_{A,\mathbf{u}}$ is box-$\mathbf n$-convex and it is of the form \eqref{eq:repres_v10} with $W_{\boldsymbol\alpha}=0$ and $\mu=\delta_{(u_1, \ldots, u_d)}$. Then by  \eqref{eq:eq100}, taking $f=f_{A,\mathbf{u}}$, we obtain \eqref{eq:eq102}.

We pass to the proof of the implication ($\Leftarrow$). Let $f\colon\mathbf I\to\mathbb R$ be a continuous box-$\mathbf n$-convex function.
Let $\boldsymbol\alpha=(\alpha_1,\alpha_2,\dots,\alpha_d)\in\mathbf I$.
Then, by Theorem \ref{th:repres_int_v3}, the function $f$ is of the form \eqref{eq:repres_v10}. Then
\begin{multline}\label{eq:repres_103}
\E f(X_1,\dots,X_d)=\E W_{\boldsymbol\alpha}(X_1,\dots,X_d)\\
+\sum_{A\subset\{1,2,\ldots,d\}}
	\idotsint\limits_{\mathbf {I_{A,\alpha}}}
	\E \prod_{j\in A'}\frac{(X_j-u_j)_+^{n_j-1}}{(n_j-1)!}\prod_{j\in A}(-1)^{n_j}\frac{(X_j-u_j)_-^{n_j-1}}{(n_j-1)!}\ d\mu(u_1,\dots,u_d).
\end{multline}
Inequality \eqref{eq:eq100} now follows from \eqref{eq:eq101}, \eqref{eq:eq102}, using identity \eqref{eq:repres_103}, and a similar identity involving $(Y_1,\dots,Y_d)$. Thus inequality \eqref{eq:eq100a} is satisfied. 

\end{proof}
By analogy to box-$\mathbf{n}$-convex orders for random vectors, we define the box-$\mathbf{n}$-convex order for signed measures $\gamma_1$, $\gamma_2$ on $\mathbf I$, which have a finite variation:  $\gamma_1\prec_{box\mhyphen \mathbf n\mhyphen cx}\gamma_2$ if

$$
\idotsint\limits_\mathbf{I} f(x_1,\dots,x_d)d\gamma_1(x_1,\dots,x_d)\leq
\idotsint\limits_\mathbf{I}  f(x_1,\dots,x_d)d\gamma_2(x_1,\dots,x_d)
$$
for all continuous box-$\mathbf n$-convex functions $f\colon \mathbf{I }\to\mathbb R$, provided the integrals exist. 
From Theorem \ref{th:theorem106}, it follows the following characterization of box-$\mathbf{n}$-convex order for signed measures.
\begin{Thm}\label{th:theorem107}
Let $\gamma$ be a signed measure on $\mathbf{I}$ which has a finite variation and such that\\$\idotsint\limits_\mathbf{I} |x_1|^{n_1-1},\dots,|x_d|^{n_d-1}d|\gamma|(x_1,\dots,x_d)<\infty$. Then
$$
\idotsint\limits_\mathbf{I} f(x_1,\dots,x_d)\:d\gamma(x_1,\dots,x_d)\geq 0
$$
for all continuous box-$\mathbf n$-convex functions $f\colon \mathbf{I }\to\mathbb R$ if, and only if,
\begin{itemize}
\item[{\rm\textbf{a)}}]
\begin{equation}\label{eq:eqth107a}
\idotsint\limits_\mathbf{I} W(x_1,\dots,x_d)\:d\gamma(x_1,\dots,x_d)=0
\end{equation}
for all continuous pseudo-polynomials $W$ of order $(n_1-1, \ldots n_d-1)$, provided the integral exists,
\item[{\rm\textbf{b)}}]

\begin{equation}\label{eq:eq102a}
\idotsint\limits_\mathbf{I}
\prod_{j\in A'}\frac{(x_j-u_j)_+^{n_j-1}}{(n_j-1)!}\prod_{j\in A}(-1)^{n_j}\frac{(x_j-u_j)_-^{n_j-1}}{(n_j-1)!}
\:d\gamma(x_1,\dots,x_d)\geq 0
\end{equation}
for all $A\subset\{1,2,\ldots,d\}$ and $(u_1, \ldots, u_d)\in \mathbf{I}$.

\end {itemize}
\end{Thm}
By Theorem \ref{th:theorem107}, we obtain the following characterization of box-$\mathbf{n}$-convex order for the signed measure $\gamma=\gamma_1\otimes, \ldots,\otimes \gamma_d$, which is a product measure of $\gamma_1,\ldots ,\gamma_d$.

\begin{Thm}\label{thm:theorem108}
Let $\gamma_i$ be  non zero signed finite Borel measures on $I_i$, such that $\int_{I_i}|x_i|^{n_i-1}\:d|\gamma_i|(x_i)<\infty$, $i=1,\ldots,d$. Then
\begin{equation}\label{eq:boxint2bis}
\idotsint\limits_\mathbf{I} f(x_1,\dots,x_d)\:d\gamma_1(x_1)\ldots\:d\gamma_d(x_d)\geq 0
\end{equation}
for all continuous  box-$\mathbf n$-convex functions $f\colon \mathbf{I }\to\mathbb R$ 
 (provided the integral exists) if, and only if , 
\begin{itemize}
\item[{\rm\textbf{a)}}]
for all $j=1,\ldots,d$ either $\gamma_j\geq_{(n_j-1)\mhyphen cx}0$ or $\gamma_j\geq_{(n_j-1)\mhyphen cv}0$, 
\item[{\rm\textbf{b)}}]
the number of those $j$ for which $\gamma_j\geq_{(n_j-1)\mhyphen cv}0$ is even.
\end{itemize}
\end{Thm}
\begin{proof}
We show the implication ($\Rightarrow$).
By Theorem \ref{th:theorem107} (a)
$$
\idotsint\limits_\mathbf{I} W(x_1,\dots,x_d)\:\:d\tau_1(x_1)\ldots\:d\gamma_d(x_d)=0
$$
for all continuous pseudo-polynomials of degree $(n_1-1, \ldots, n_d-1) $ of the form
\begin{equation}\label{eq:eq105c}
W(x_1,x_2,\dots,x_d)=\sum_{i=1}^d\sum_{k=0}^{n_i-1}A_{ik}(\mathbf x_{\{i\}'})x_i^k,
\end{equation}
provided the integral exists.
Then, for all $i=1,\ldots, d$ and $k=0,\ldots, n_i-1$
\begin{equation}\label{eq:eq105b}
\int_{I_i}x_i^k\:d \gamma_i(x_i)\idotsint\limits_{I_1\times\ldots\times I_{i-1}\times I_{i-1}\times\ldots\times I_d} A_{ik}(\mathbf x_{\{i\}'})\:d\gamma_1(x_1)\ldots\:d\gamma_{i-1}(x_{i-1})\ d\gamma_{i+1}(x_1)\ldots\:d\gamma_d (x_d)=0.
\end{equation}
Since $\gamma _i$, $i=1, \ldots, d$ are non-zero signed measures, it follows that there exists the functions $A_{ik}$, such that the second integral in the above expression is non-zero. Then by \eqref{eq:eq105b}, we obtain
\begin{equation}\label{eq:eq106}
\int_{I_i}x_i^k\:d\gamma_i(x_i)=0, \quad i=1,\ldots, d, \:k=0,1,\ldots, n_i-1.
\end{equation}
By \eqref{eq:eq102a}, for $A=\emptyset$, we have
\begin{equation}\label{eq:eq107}
\prod_{j=1}^{d}\int_{I_j}\frac{(x_j-u_j)_+^{n_j-1}}{(n_j-1)!}\:d\gamma_j(x_j)\geq 0
\end{equation}
for all $(u_1, \ldots, u_d)\in \mathbf{I}$. We denote
$$
H(\gamma_j,u_j) =\int_{I_j}\frac{(x_j-u_j)_+^{n_j-1}}{(n_j-1)!}\:d\gamma_j(x_j).
$$
Since $\gamma _i$, $i=1, \ldots, d$, are non-zero signed measures, it follows that there exist $(v_1, \ldots, v_d)\in \mathbf{I}$ such that 
$\label{eq:eq107a}
H(\gamma_i,v_i)\neq0$
for $i=1,\dots,d$.

Let $i=1,\ldots, d$. By \eqref{eq:eq107}, $H(\gamma_i,u_i)\prod_{j\neq i}H(\gamma_j,v_j)\geq0$ for each $u_i\in I_i$. Consequently, either $H(\gamma_i,u_i)\geq 0$ for all $u_i\in I_i$ or $H(\gamma_i,u_i)\leq 0$ for all $u_i\in I_i$. 
 
Moreover, by \eqref{eq:eq107}, we conclude, that the number of $j$'s such that $H(\gamma_j,u_j)\leq 0$ for all $u_j\in I_j$, is even. Taking into account \eqref{eq:eq106}, by Theorem \ref{th:theorem103} and Remark \ref{rem:rem105}, ($\Rightarrow$) is proved.

We pass to the proof of the implication ($\Leftarrow$).
Assume, that conditions a) and b) are satisfied. By Theorem \ref{th:theorem103},
we have that \eqref{eq:eq106} is satisfied, then taking into account \eqref{eq:eq105b}, \eqref{eq:eq105c}, we obtain \eqref{eq:eqth107a} with $\gamma=\gamma_1\otimes, \ldots,\otimes \gamma_d$.

By Theorem \ref{th:theorem103} and Remark  \ref{rem:rem105}
 we have that for all $j=1,\ldots, d$, either $H(\gamma_j,u_j)\geq 0$ for all $u_j\in I_j$ or $H(\gamma_j,u_j)\leq 0$ for all $u_j\in I_j$, and the number of those $j$ for which $H(\gamma_j,u_j)\leq 0$ for all $u_j\in I_j$, is even. 
 
By Remark \ref{rem:rem104}, we conclude that the function $H(\gamma_j,u_j)$ has the same sign as the function
$$S(\gamma_j,u_j)=\int_{I_j}(-1)^{n_j-1}\frac{(x_j-u_j)_+^{n_j-1}}{(n_j-1)!}\:d\gamma_j(x_j).$$
This implies that
for each $A\subset\{1,\dots,d\}$ and $(u_1,\dots,u_d)\in\mathbf I$, we have
$$\prod_{j\in A'}H(\gamma_j,u_j)\prod_{j\in A}S(\gamma_j,u_j)\geq0.$$
Consequently, \eqref{eq:eq102a} is satisfied for $\gamma=\gamma_1\otimes, \ldots,\otimes \gamma_d$. By Theorem \ref{th:theorem107}, this completes the proof of ($\Leftarrow$).
\end{proof}

By Theorem \ref{th:theorem103} and Remark \ref{rem:rem104}, we obtain that Theorem \ref{thm:theorem108} can be rewritten in the following form.

\begin{Thm}\label{thm:theorem37a}
Let $\gamma_i$ be  non zero signed finite Borel measures on $I_i$, such that $\int_{I_i}|x_i|^{n_i-1}\:d|\gamma_i|(x_i)<\infty$, $i=1,\ldots,d$. Then
\begin{equation*}
\idotsint\limits_\mathbf{I} f(x_1,\dots,x_d)\:d\gamma_1(x_1)\ldots\:d\gamma_d(x_d)\geq 0
\end{equation*}
for all continuous  box-$\mathbf n$-convex functions $f\colon \mathbf{I }\to\mathbb R$ (provided the integral exists) if, and only if , 
\begin{itemize}
\item[{\rm\textbf{a)}}]
\begin{equation}\label{eq:eq106bis}
\int_{I_i}x_i^k\:d\gamma_i(x_i)=0, \quad i=1,\ldots, d, \:k=0,1,\ldots, n_i-1.
\end{equation}
\item[{\rm\textbf{b)}}]
\begin{equation*}\label{eq:eq107bis}
\prod_{j=1}^{d}\int_{I_j}\frac{(x_j-u_j)_+^{n_j-1}}{(n_j-1)!}\:d\gamma_j(x_j)\geq 0
\end{equation*}
for all $(u_1, \ldots, u_d)\in \mathbf{I}$.
\end{itemize}
\end{Thm}

From Theorem \ref{thm:theorem108}, it follows immediately the following theorem.
\begin{Thm}\label{thm:thm28bis}
Let $\mu_i$, $\nu_i$ be probability measures on $I_i$ such that $\int_{I_i}|x_i|^{n_i-1}$ $d\mu_i(x_i)<\infty$, $\int_{I_i}|x_i|^{n_i-1}\:d\nu_i(x_i)<\infty$ and $\nu_i\leq _{(n_j-1)\mhyphen cx}\mu_i$, $i=1,\ldots,d$.  Then
\begin{equation*}\label{eq:boxint2tres}
\idotsint\limits_\mathbf{I} f(x_1,\dots,x_d)\:d(\mu_1-\nu_1)(x_1)\ldots\:d(\mu_d-\nu_d)(x_d)\geq 0
\end{equation*}
for all continuous  box-$\mathbf n$-convex functions $f\colon \mathbf{I }\to\mathbb R$ (provided the integral exists).
\end{Thm}
\section{The Hermite-Hadamard, Jensen and Ra\c{s}a inequalities}

We recall the classical Hermite-Hadamard and Jensen inequalities.

\begin{prop}\label{prop:prop25}
Let $f \colon I\to\mathbb{R}$ be a convex function defined on a real interval $I$ and $a,b \in I$ with $a<b$. The following double inequality
\begin{equation*}
f\left(\frac{a+b}2\right)\leq\frac1{b-a}\cdot\int_a^b f(x)\,dx\leq\frac{f(a)+f(b)}2 \label{eq:0bis}
\end{equation*}
is known as the Hermite-Hadamard inequality for convex functions (see \cite{DraPe2000}), which is equivalent to the convex orde\-ring relations (see \cite{rajba2014})
\begin{equation*}\label{eq:eq25}
\delta_{(a+b)/2}\leq_{1\mhyphen cx}\:\frac1{b-a}\chi _{[a,b]}(x)\, dx\leq_{1\mhyphen cx}\: \frac{1}{2}(\delta_a+\delta_b).
\end{equation*}
\end{prop}
\begin{prop}\label{prop:prop25a} 
One of the most familiar and elementary inequalities in the probability theory is the Jensen inequality:
\begin{equation}\label{eq:conv_prob}
 f\bigl(\E X \bigr) \leq \E f(X),
\end{equation}
where the function $f$ is convex over the convex hull of the range of the random variable~$X$ (see \cite{Bil95}). Inequality \eqref{eq:conv_prob} is equivalent to the convex orde\-ring relation (see \cite{rajba2014})
\begin{equation*}\label{eq:eq25a}
\delta_{\E X}\leq_{1\mhyphen cx}\:\mu _X.
\end{equation*}
\end{prop}
In this paper, we give some Hermite-Hadamard, Jensen and Ra\c{s}a inequalities for box-$\mathbf{n}$-convex functions.

Note that Theorem \ref{thm:thm28bis} is equivalent to the following theorem.
\begin{Thm}\label{thm:thm28bisa}
Let $X_i$, $Y_i$ be $I_i$ valued random variables such that $\E |X_i |^{n_{i-1}}< \infty$, $\E |Y_i |^{n_{i-1}}< \infty$ and $X_i\leq _{(n_j-1)\mhyphen cx}Y_i$, $i=1,\ldots,d$. For each $A\subset \{1,\ldots,d\}$, by $Z_{1,A},\ldots, Z_{d,A}$ we denote independent random variables such that
\begin{equation}\label{eq:eqmuz}
\mu_{Z_{i,A}}=
 \left\{
\begin{aligned}
 \mu_{Y_i}\quad&\ \text{if }i\notin A,\\
 \mu_{X_i}\quad&\ \text{if } i \in A.
\end{aligned}
\right.
 \end{equation}
Then
\begin{equation}\label{eq_eqhermite}
\sum_{ A\subset \{1,\ldots , d  \}} (-1)^{|A|} \E f(Z_{1,A}, \ldots, Z_{d,A} )\geq 0, 
\end{equation}
for all continuous box-$\mathbf{n} $-convex functions $f\colon \mathbf{I }\to\mathbb R$ (provided the expectations exist).
\end{Thm}
From Proposition \ref{prop:prop25} and Theorem \ref{thm:thm28bisa}, we obtain the following Hermite-Hadamard type inequalities for box-$(2,\ldots,2)$-convex functions. 
\begin{Thm}
Let $a_i,b_i \in I_i$ with $a_i<b_i$, $i=1,\ldots,d$.  
Then, for all continuous box-$(2,\ldots,2)$-convex functions $f\colon \mathbf{I }\to\mathbb R$,
\begin{itemize}
\item[{\rm\textbf{(a)}}]
 \textbf{the first Hermite-Hadamard inequality} is given by \eqref{eq_eqhermite}, with 
 $$
\mu_{X_i}=\delta_{(a_i+b_i)/2}\quad \mu_{Y_i}=\frac1{b_i-a_i}\chi _{[a_i,b_i]}(x)\, dx,\quad (i=1, \ldots,d),
 $$
 \item[{\rm\textbf{(b)}}]
\textbf{the second Hermite-Hadamard inequality}is given by \eqref{eq_eqhermite}, with 
 $$
 \mu_{X_i}=\frac1{b_i-a_i}\chi _{[a_i,b_i]}(x)\, dx,\quad \mu_{Y_i}=\frac{1}{2}(\delta_{a_i}+\delta_{b_i})\quad(i=1, \ldots,d).
 $$

\end{itemize}
\end{Thm} 

For $d=3$, we obtain the following first Hermite-Hadamard inequality.
\begin{Thm}
Let $a_i,b_i \in I_i$ with $a_i<b_i$, $i=1,2,3$.  
Then, for all continuous box-$(2,2,2)$-convex functions $f\colon \mathbf{I }\to\mathbb R$, we have 

 \textbf{the first Hermite-Hadamard inequality}
\begin{multline*}
\frac1{(b_1-a_1)(b_2-a_2)(b_3-a_3)}\int_{a_1}^{b_1}\int_{a_2}^{b_2}\int_{a_3}^{b_3}f(x,y,z)dz\,dy\,dx
-\frac1{(b_2-a_2)(b_3-a_3)}\int_{a_2}^{b_2}\int_{a_3}^{b_3}f\left(\frac{a_1+b_1}{2},y,z\right)dz\,dy\\\\
-\frac1{(b_1-a_1)(b_3-a_3)}\int_{a_1}^{b_1}\int_{a_3}^{b_3}f\left(x,\frac{a_2+b_2}{2},z\right)dz\,dx
-\frac1{(b_1-a_1)(b_2-a_2)}\int_{a_1}^{b_1}\int_{a_2}^{b_2}f\left(x,y,\frac{a_3+b_3}{2}\right)dy\,dx\\\\
+\frac1{(b_1-a_1)}\int_{a_1}^{b_1}f\left(x,\frac{a_2+b_2}{2},\frac{a_3+b_3}{2}\right) dx
+\frac1{(b_2-a_2)}\int_{a_2}^{b_2}f\left(\frac{a_1+b_1}{2},y,\frac{a_3+b_3}{2}\right)dy\\\\
+\frac1{(b_3-a_3)}\int_{a_3}^{b_3}f\left(\frac{a_1+b_1}{2},\frac{a_2+b_2}{2},z\right)dz-f\left(\frac{a_1+b_1}{2},\frac{a_2+b_2}{2},\frac{a_3+b_3}{2}\right)\geq 0.
\end{multline*}
\end{Thm}

From Proposition \ref{prop:prop25} and Theorem \ref{thm:thm28bisa}, we obtain the following Jensen type inequalities for box-$\mathbf{n} $-convex functions.
\begin{Thm}\label{thm:thmJensen}
Let $X_i$ be $I_i$ valued random variables, $i=1,\ldots,d$.  
Then, for all continuous box-$(2,\ldots,2)$-convex functions $f\colon \mathbf{I }\to\mathbb R$, we have \textbf{the Jensen inequality}
\begin{equation*}
\sum_{ A\subset \{1,\ldots , d  \}} (-1)^{|A|} \E f(Z_{1,A}, \ldots, Z_{d,A} )\geq 0, 
\end{equation*}
where $Z_{1,A}, \ldots, Z_{d,A}  $ are independent random variables such that
\begin{equation*}
\mu_{Z_{i,A}}=
 \left\{
\begin{aligned}
 \mu _{X_i}\quad&\ \text{if }i\notin A,\\
 \delta_{\:\E X_i}\quad&\ \text{if } i \in A.
\end{aligned}
\right.
 \end{equation*}

\end{Thm} 
For $d=3$, we obtain the following Jensen inequality.
\begin{Thm}\label{thm:thmJensen3}
Let $X,Y,Z$ be $I_1,I_2,I_3$ valued random variables, respectively.  
Then, for all continuous box-$(2,2,2)$-convex functions $f\colon \mathbf{I }\to\mathbb R$, we have \textbf{the Jensen inequality}
\begin{multline*}
\E f(X,Y,Z )
-\E f(\E X,Y,Z )-\E f(X,\E Y,Z )-\E f(X,Y,\E Z )\\
+\E f(X,\E Y,\E Z )+\E f(\E X ,Y,\E Z )+\E f(\E X,\E Y,Z )-f(\E X,\E Y,\E Z )\geq 0.
\end{multline*}
\end{Thm}
If the random variables have discrete distributions, then the above Jensen inequality has the following form.
\begin{Thm}\label{thm:thmJensen2}
Let $I_1,I_2,I_3$ be real intervals. Let $M,N,L>1$. $x_1,\ldots, x_M\in I_1$, $y_1,\ldots, y_N\in I_2$, $z_1,\ldots, z_L\in I_3$, 
$\alpha_1,\ldots, \alpha_M,\beta_1,\ldots, \beta_N,\gamma_1,\ldots, \gamma_L\geq 0$, 
$\alpha_1+\ldots+\alpha_M=\beta_1+\ldots+\beta_N=\gamma_1+\ldots + \gamma_L=1$, 
$\bar{x}=\alpha_1 x_1+\ldots+\alpha_M x_M$, $\bar{y}=\beta_1 y_1+\ldots+\beta_N y_N$, $\bar{z}=\gamma_1 z_1+\ldots + \gamma_L z_L$.  Then for all continuous box-$(2,2,2)$-convex functions
\begin{multline*}
\sum_{i=1}^{M}\sum_{j=1}^{N}\sum_{k=1}^{L}\alpha_i \beta_j \gamma_k f(x_i ,y_j,z_k)
-\sum_{j=1}^{N}\sum_{k=1}^{L}\beta_j \gamma_k f(\bar{x} ,y_j,z_k)-\sum_{i=1}^{M}\sum_{k=1}^{L}\alpha_i  \gamma_k f(x_i ,\bar{y},z_k)
-\sum_{i=1}^{M}\sum_{j=1}^{N}\alpha_i \beta_j  f(x_i ,y_j,\bar{z})
\\
+\sum_{i=1}^{M}\alpha_i   f(x_i ,\bar{y},\bar{z})
+\sum_{j=1}^{N}\beta_j  f(\bar{x} ,y_j,\bar{z})+\sum_{k=1}^{L} \gamma_k f(\bar{x} ,\bar{y},z_k)-f(\bar{x} ,\bar{y},\bar{z})\geq 0.
\end{multline*}
\end{Thm}
Now we recall the Ra\c sa inequality. Its probabilistic version has the following form (see \cite{Rasa2014b}, \cite{MRW2017}) 
\begin{equation*}\label{eq:box17}
(\nu-\mu)^{*2}\geq_{1\mhyphen cx} 0 , 
\end{equation*}
where $\mu=B(n,x)$ and $\nu=B(n,y)$ are the binomial distributions with parameters $n\in\mathbb N$ and $x,y\in[0,1]$. By $(\nu-\mu)^{*2}$ we mean $(\nu-\mu)*(\nu-\mu)$, where $*$ is the convolution of signed measures. In \cite{KomRaj2018,KomRaj2021,KomRaj2022}, we obtained some useful necessary and sufficient conditions for Borel measures $\mu$ and $\nu$ to satisfy
the following generali\-zed Ra\c{s}a inequality, involving $q$-th convolutional power of the signed measure $\nu-\mu$:
$$(\nu-\mu)^{*q}\geq_{(q-1)\mhyphen cx} 0 , \quad q\geq 2.$$
In \cite{KomRaj2023}, we investigated the generali\-zed Ra\c{s}a inequality for box-$(m,n) $-convex fun\-ctions. In this paper, we study  the generali\-zed Ra\c{s}a inequality for box-$\mathbf{n}$-convex functions.
We will need two lemmas. 
\begin{Lem}[\cite{KomRaj2022}, Lemma 5, p.5]\label{lem:box7}
Let $\tau_1, \ldots ,\tau_n$ be signed measures on $\R$ with finite variation, such that $\tau_i(\R)=0$ and $\int_{-\infty}^{\infty} |x|^{n-1} |\tau_i|(dx)<\infty$, $i=1,\ldots, n$, Then
\begin{itemize}
\item[{\rm\textbf{(a)}}] 
$\tau_1* \ldots *\tau_n (\R)=0 $,
\medskip
\item[{\rm\textbf{(b)}}]
$\int_{-\infty}^{\infty} x^k \tau_1* \ldots *\tau_{n} (dx)=0$ for all integers $0<k<n$.\end{itemize}
\end{Lem}
\begin{Lem}[\cite{KomRaj2022}, p.7]\label{lem:box8}
Let $\tau_1, \ldots ,\tau_q$ be signed measures on $\R$ with finite variation, such that $\tau_i(\R)$   $=0$, $i=1,\ldots, q$.
 Then for all $A\in \R$
\begin{equation*}\label{eq:lembox8}
\int_{-\infty}^{\infty} \frac{\bigl(x-A\bigr)^{q-1}_+ }{(q-1)!}\: \tau_1*\ldots*\tau_q(dx)=
\overline F_{\tau_1} * \overline F_{\tau_2}*\ldots * \overline F_{\tau_q}(A),
\end{equation*}
where $\overline F_{\tau_i}(x)=\tau_i([x,\infty))$, $i=1,\ldots, q$, $x\in \R$.
On the right side of the above equality, the symbol $*$ denotes the convolution of functions on $\mathbb R$.
\end{Lem}
\begin{Thm}\label{th:mn_int_19tres}
Let $\mathbf{n} =n_1, \ldots, n_d$, $n_i\geq 2$, $ i=1,\ldots, d$. Let $\mu_i,\nu_i, $ be  probability measures on $\R$, such that $\int_{-\infty}^{\infty} |x|^{n_i-1} \mu_i(dx)<\infty$,
$\int_{-\infty}^{\infty} |x|^{n_i-1} \nu_i(dx)<\infty$ for $i=1,\dots,d$,
Then the following conditions are equivalent:
\begin{itemize}
\item[{\rm\textbf{a)}}] \textbf{the Ra\c{s}a inequality} for continuous box-$\mathbf{n}$-convex functions
$$(\nu_1-\mu_1)^{*n_1}\otimes\ldots\otimes(\nu_d-\mu_d)
^{*n_d}
\succ_{box\mhyphen \mathbf{n}\mhyphen cx}   0,$$

\item[{\rm\textbf{b)}}]

$$
(\overline{F}_{\nu_1}- \overline{F}_{\mu_1} )^{*n_1}(A_1)
\times\ldots\times
(\overline{F}_{\nu_m}- \overline{F}_{\mu_m} )^{*n_d}(A_d)
\geq  0
$$
for all $A_i\in \R$, $i=1,\dots,d$.
\end{itemize} 
\end{Thm}
\begin{proof}
Let $\tau_i=\nu_i-\mu_i$, $i=1,\ldots,d$. Then $\tau_i$  are signed measure such that $\tau_i(\R)=0$, $i=1,\ldots, d$,  thus by Lemma \ref{lem:box7}
, it follows that for $\gamma_i=\tau_i^{*n_i} $
$$\int_{-\infty}^{\infty} x_i^k \gamma_{i} (dx_i)=0, \quad  k=0,1,\ldots, n_i-1,\quad i=1,\ldots, d.$$

Then, by Theorem \ref{thm:theorem37a}, we conclude that condition a)
is equivalent to 
\begin{equation*}\label{eq:eq107bisbis}
\prod_{j=1}^{d}\int_{I_j}\frac{(x_j-A_j)_+^{n_j-1}}{(n_j-1)!}\:d\gamma_j(x_j)\geq 0
\end{equation*}
for all $(A_1, \ldots, A_d)\in \mathbf{I}$. Taking into account, that by Lemma \ref{lem:box8},
$$  
\int_{I_j}\frac{(x_j-A_j)_+^{n_j-1}}{(n_j-1)!}\:d\gamma_j(x_j)=(\overline F_{\tau_j})^{*n_j}(A_j),\quad j=1,\ldots, d,
$$
the theorem is proved.
\end{proof}

Similarly to \cite{KomRaj2023}, we define strongly box-$\mathbf{n} $-convex functions.
\begin{defin}
We say that a function $f\colon\mathbf I\to\mathbb R$ is strongly box-$\mathbf{n} $-convex functions with modulus $C\geq 0 $, if the function $g\colon\mathbf I\to\mathbb R$ given by the formula $g(x_1, \ldots, x_d) =f(x_1, \ldots, x_d)-C x_1^{n_1} \ldots x_d^{n_d}$ $((x_1, \ldots, x_d)\in  \mathbf I)$ is box-$\mathbf{n} $-convex.
\end{defin}

Then, using the methods presented above and in \cite{KomRaj2023}, one can prove other the Ra\c{s}a, Jensen and Hermite-Hadamard inequalities for box-$\mathbf{n} $-convex functions as well as for strongly box-$\mathbf{n} $-convex functions.

\end{document}